\documentclass[11pt]{article}
\usepackage{color}
\usepackage{graphicx}
\usepackage{amsmath,amsthm}
\usepackage{amssymb,mathrsfs}
\usepackage{geometry}
\usepackage{verbatim}
\usepackage[format=hang,font=small,textfont=it]{caption}
\usepackage[nottoc]{tocbibind}
\usepackage{pgfplots}
\pgfplotsset{compat=1.15}
\usepackage{mathrsfs}
\usetikzlibrary{arrows}
\usepackage{hyperref}
\hypersetup{colorlinks, linkcolor=blue}

\newtheorem{thm}{Theorem}[section]
\newtheorem{lemma}[thm]{Lemma}
\newtheorem{prop}[thm]{Propsition}

\theoremstyle{definition}
\newtheorem{definition}[thm]{Definition}
\newtheorem{example}[thm]{Example}

\newtheorem{remark}[thm]{Remark}
\newtheorem{question}[]{Question}

\newcommand\ddd{\mathrm{d}}
\newcommand\holder{\mathrm{H\ddot{o}lder}}

\newcommand\bR{\mathbb{R}}
\newcommand\bQ{\mathbb{Q}}
\newcommand\bN{\mathbb{N}}
\newcommand\bZ{\mathbb{Z}}

\def \l {\left}
\def \r {\right}

\begin{document}
\renewcommand{\thefootnote}{\fnsymbol {footnote}}
\title{Two-weight Norm Inequalities for Local Fractional Integrals on Gaussian Measure Spaces}
\footnotetext{{}{2010 \emph{Mathematics Subject Classification}: Primary 42B35; Secondary 42B20, 42B25.}} \footnotetext{{}\emph{Key words and phrases}: Local fractional integral, local fractional maximal operator,  two-weight inequality, Gaussian measure space.}
\setcounter{footnote}{0}\author{Boning Di, Qianjun He\footnote{Correspoding author}, Dunyan Yan}
\date{}
\maketitle

\begin{abstract}
In this paper, the authors establish the two-weight boundedness of the local fractional maximal operators and local fractional integrals on Gaussian measure spaces associated with the local weights. More precisely, the authors first obtain the two-weight weak-type estimate for the local-$a$ fractional maximal operators of order $\alpha$ from $L^{p}(v)$ to $L^{q,\infty}(u)$ with $1\leq p\leq q<\infty$ under a condition of $(u,v)\in \bigcup_{b'>a} A_{p,q,\alpha}^{b'}$, and then obtain the two-weight weak-type estimate for the local fractional integrals. In addition, the authors obtain the two-weight strong-type boundedness of the local fractional maximal operators under a condition of $(u,v)\in\mathscr{M}_{p,q,\alpha}^{6a+9\sqrt{d}a^2}$ and the two-weight strong-type boundedness of the local fractional integrals. These estimates are established by the radialization method and dyadic approach.
\end{abstract}

\section{Introduction}\label{DHY_2101_introduction}
The Gaussian measure space, denoted by $\l(\bR^d, |\cdot|, \gamma\r)$, is the Euclidean space $\bR^d$ endowed with the Euclidean distance $|\cdot|$ and the Gaussian probability measure $\gamma$, where
\[\ddd\gamma(x):=\pi^{-d/2} e^{-|x|^2} \ddd x.\]
The Gaussian harmonic analysis extends the classical results obtained in harmonic analysis of trigonometric expansions to orthogonal polynomial expansions~\cite{MS_1965}. On the other hand, Gaussian harmonic analysis is widely used in the second quantization~\cite{Nelson_1973}, Malliavin calculus~\cite{Nualart_2006}, hypercontractivity~\cite{BE_1985}, and geometric applications\cite{Bakry_1988}, etc.

On Gaussian measure spaces, due to the special locally doubling and reverse doubling property which will be shown in Section \ref{DHY_2101_preliminaries}, we concentrate on the \textit{local Hardy-Littlewood maximal operator $M^{a}$} defined by
\[M^{a}(f)(x):=\sup_{Q\in\mathscr{Q}_a(x)} \frac{1}{\gamma(Q)} \int_Q |f(y)|\ddd \gamma(y),\]
where
\[\mathscr{Q}_a(x)=\{Q\in\mathscr{Q}_a: Q\ni x\}, \quad \mathscr{Q}_a=\{Q\subset \bR^d: \ell(Q)\leq am(c_Q)\}, \quad  m(x)=\min\{1, {1}/{|x|}\}\]
with $c_Q$ denoting the center of the cube $Q$ and $\ell(Q)$ denoting the side length of the cube $Q$. Furthermore, the \textit{local fractional Hardy-Littlewood maximal operator $M_{\alpha}^{a}$} is defined by
\[M^{a}_{\alpha}(f)(x):=\sup_{Q\in\mathscr{Q}_a(x)}\frac{1}{[\gamma(Q)]^{1-\dot{\alpha}}}\int_Q |f(y)|\ddd \gamma(y)\]
with $\dot{\alpha}=\alpha/d$. Here all the sides of the cubes are parallel to the coordinate axes.

Recently the study of Gaussian harmonic analysis has aroused extensive attention. In 2007, Mauceri and Meda~\cite{MM_2007} introduced the local maximal operator and used it to develop the singular integral operator theory on the Gaussian measure spaces; in the same year, Aimar~et~al.~\cite{AFS_2007} obtained the weak type $(1,1)$ inequalities for the Gaussian Riesz transform and a general maximal operator which dominates the Ornstein-Uhlenbeck maximal operator; then in 2010, Liu and Yang\cite{LY_2010} obtained the boundedness of the \textit{local fractional integral operator} defined by
\[I_{\alpha}^{a}(f)(x):=\int_{B(x, am(x))} \frac{f(y)}{[\gamma(B(x,|x-y|))]^{1-\dot{\alpha}}}\ddd\gamma(y)\]
and the local fractional maximal operator $M^{a}_{\alpha}$ on the Gaussian Lebesgue spaces; in 2014, Liu~et~al.~\cite{LSY_2014} obtained the boundedness of the local fractional integral operator and the local maximal operator on the Gaussian Morrey-type spaces; later in 2016, Wang~et~al.~\cite{WZL_2016} characterized the one-weight boundedness of $M^a$ on Gaussian Lebesgue spaces by the \textit{local $A_{p}^a$ weights} defined as
\[[\omega]_{A_{p}^a}=\sup_{B\in\mathscr{B}_{a}}\l(\frac{1}{\gamma(B)}\int_B \omega(x) \ddd\gamma(x)\r)\l(\frac{1}{\gamma(B)}\int_B \omega(x)^{1-p'}\ddd\gamma(x)\r)^{p-1}<\infty,\]
and obtained the one-weight boundedness of $M_{\alpha}^a$ on the Gaussian Lebesgue spaces by the \textit{local $A_{p,q}^a$ weights} defined as
\begin{equation*}
[w]_{A_{p,q}^a}=\sup_{B\in\mathscr{B}_{a}} \l(\frac{1}{\gamma(B)}\int_B \omega(x)^q\ddd \gamma(x)\r)^{1/q}\l(\frac{1}{\gamma(B)}\int_B \omega(x)^{-p'}\ddd \gamma(x)\r)^{1/{p'}}<\infty;
\end{equation*}
more recently in 2020,  Lin and Mao~\cite{LM_2020} established the one-weight norm inequalities associated with the local $A_{p,q}^a$ weights for the fractional operators $I_{\alpha}^a$ and $M_{\alpha}^a$ on the Gaussian measure spaces.

Based on the results above, this article aims to establish the two-weight norm inequalities for the local fractional maximal and integral operators with respect to the local Muckenhoupt type weights and the local Sawyer type weights. To state our main results, we introduce the following local Muckenhoupt type weights.
\begin{definition}
Given $0<a<\infty$, $0\leq\dot{\alpha}<1$, $\alpha=\dot{\alpha}d$ and $1< p, q<\infty$, we say that a pair of weights $(u,v)\in A_{p,q,\alpha}^a$ if
\begin{align*}
[u,v]_{A_{p,q,\alpha}^a}:=\sup_{Q\in\mathscr{Q}_a} \gamma(Q)^{(\dot{\alpha}+1/q-1/p)}\l(\frac{1}{\gamma(Q)} \int_Q u(x) \ddd \gamma(x)\r)^{1/q}\l(\frac{1}{\gamma(Q)} \int_Q v(x)^{1-p'}\ddd\gamma(x)\r)^{1/{p'}}<\infty.
\end{align*}
For $p=1$, we say $(u,v)\in A_{1,q,\alpha}^a$ if
\[[u,v]_{A_{1,q,\alpha}^a}:=\sup_{Q\in\mathscr{Q}_a} \mathop{\mathrm{ess}\sup}\limits_{x\in Q}^{}\gamma(Q)^{(\dot{\alpha}+1/q-1)} \l(\frac{1}{\gamma(Q)} \int_Q  u(x) \ddd \gamma(x)\r)^{1/q} v(x)^{-1}<\infty,\]
where the essential supremum is associated with the measure $\gamma$.
\end{definition}

\begin{remark}\label{DHY_1}
Note that if $\alpha+d/q-d/p=0$ and
\[u(x)=\omega(x)^q,\quad v(x)=\omega(x)^p,\]
then the two-weight condition $A_{p,q,\alpha}^a$ goes back to the one-weight condition $A_{p,q}^a$ first introduced by Wang~et~al~\cite{WZL_2016}; if $\alpha+d/q-d/p<0$, by letting suitable $\gamma(Q)\to 0$, we conclude the fact $A_{p,q,\alpha}^a=\emptyset$; finally if $\alpha+d/q-d/p>0$, it is obvious that $u(x)=v(x)=1$ satisfy the $A_{p,q,\alpha}^a$ condition, i.e., $A_{p,q,\alpha}^a\neq\emptyset$.
\end{remark}
One of our main results in this paper is the following two-weight weak-type estimate for $M_{\alpha}^a$.

\begin{thm}\label{GM_2001_thm4.3}
Given $a\in(0,\infty)$, $1\leq p\leq q<\infty$, $0\leq\alpha<d$ and a pair of weights $(u,v)$, if 
\[(u,v)\in \bigcup_{b'>a}A_{p,q,\alpha}^{b'},\]
then $M_{\alpha}^a$ is bounded from $L^p\l(\bR^d, v,\gamma\r)$ to $L^{q,\infty}\l(\bR^d, u,\gamma\r)$, that is, there exists a constant $C>0$ such that
\[\int_{\{x\in\bR^d: M_{\alpha}^a(f)(x)>\lambda\}}u(x)\ddd\gamma(x)\leq \frac{C}{\lambda^q}\l(\int_{\bR^d}|f(x)|^p v(x)\ddd\gamma(x)\r)^{q/p}\]
holds for all $\lambda>0$.
\end{thm}
In addition, we get the following two-weight weak-type estimate for the \textit{local fractional integral opeartor $\tilde{I}_{\alpha}^a$ with cubes on Gaussian measure spaces} defined by
\[\tilde{I}_{\alpha}^a(f)(x):=\int_{Q\l(x, am(x)\r)} \frac{f(y)}{\l[\gamma\l(Q(x,|x-y|)\r)\r]^{1-\dot{\alpha}}} \ddd\gamma(y),\]
where $Q\l(x, am(x)\r)$ denotes the cube with $c_Q=x$ and $\ell(Q)=am(x)$.
\begin{thm}\label{TWWTEFI}
Given $a\in(0,\infty)$, $1\leq p\leq q<\infty$, $0<\alpha<d$ and a pair of weights $(u,v)$, if
\[(u,v)\in\bigcup_{a<b'}\bigcup_{0<\beta'<\alpha} A_{p,q,\beta'}^{b'},\]
then $\tilde{I}_{\alpha}^a$ is bounded from $L^p(\bR^d, v, \gamma)$ to $L^{q,\infty}(\bR^d, u, \gamma)$, that is, there exists a constant $C>0$ such that
\[\int_{\{x\in\bR^d: \tilde{I}_{\alpha}^a(f)(x)>\lambda\}} u(x) \ddd\gamma(x)\leq \frac{C}{\lambda^q}\l(\int_{\bR^d} |f(x)|^p v(x)\ddd\gamma(x)\r)^{q/p}\]
holds for all $\lambda>0$.
\end{thm}

Based on the radialization method in \cite{DHY_2012} and inspired by \cite{GM_2001}, we prove Theorem~\ref{GM_2001_thm4.3} by introducing a radial version of the local fractional maximal operator $M_{\alpha}^a$ and the $A_{p,q,\alpha}^a$ weights. So that we can use the dyadic analysis on Gaussian measure spaces to prove the desired conclusion. Then Theorem~\ref{TWWTEFI} comes from an extended Welland type inequality on Gaussian measure spaces (see Lemma \ref{LM_2020_lemma3.4} below) and the boundedness of $M_{\alpha}^a$. Furthermore using a similar method, we also establish the two-weight strong-type estimates for $M_{\alpha}^a$ and $\tilde{I}_{\alpha}^a$ under the following local Sawyer type condition.

\begin{definition}
Let $a\in(0,\infty)$, $\alpha\in[0,d)$ and $1<p\leq q<\infty$. We say that a pair of weights $(u,v)$ satisfies the \textit{local-$a$ testing condition} if
\[[u,v]_{\mathscr{M}_{p,q,\alpha}^a}:=\sup_{Q\in\mathscr{Q}_a} \l[\int_Q \l(M_{\alpha}^a(v^{1-p'}\chi_Q)(x)\r)^q u(x) \ddd  \gamma(x)\r]^{1/q}\l(\int_Q v(x)^{1-p'}\ddd\gamma(x)\r)^{-1/p}<\infty.\]
In this situation, we write $(u,v)\in \mathscr{M}_{p,q,\alpha}^a$.
\end{definition}

\begin{remark}
We can rewrite the definition of $[u,v]_{\mathscr{M}_{p,q,\alpha}^a}$ as
\begin{align*}
\sup_{Q\in\mathscr{Q}_a}\l[\int_Q\l(\sup_{Q'\in\mathscr{Q}_a(x)}\frac{1}{\gamma(Q')^{1-\dot{\alpha}}}\int_{Q'\cap Q} v(y)^{1-p'}\ddd\gamma'(y)\r)^q u(x)\ddd\gamma(x)\r]^{\frac{1}{q}}\l(\int_Q v(x)^{1-p'}\ddd\gamma(x)\r)^{-\frac{1}{p}}.
\end{align*}
Thereby when $d/q-d/p+\alpha<0$, similar to Remark \ref{DHY_1}, we conclude $\mathscr{M}_{p,q,\alpha}^a =\emptyset$ by letting suitable $Q=Q'$ and $\gamma(Q')=\gamma(Q)\to 0$.
\end{remark}

Then we obtain the two-weight strong-type boundedness of $M_{\alpha}^a$ and $\tilde{I}_{\alpha}^a$ as follows.
\begin{thm}\label{GM_2001_thm3.1}
Given $a\in(0,\infty)$, $1< p\leq q<\infty$, $0\leq\alpha<d$ and a pair of weights $(u,v)$, if
\[(u,v)\in \mathscr{M}_{p,q,\alpha}^{6a+9\sqrt{d}a^2},\]
then $M_{\alpha}^a$ is bounded from $L^p(\bR^d, v, \gamma)$ to $L^{q}(\bR^d, u, \gamma)$, that is, there exists a constant $C>0$ such that
\[\l(\int_{\bR^d}\Big(M_{\alpha}^a(f)(x)\Big)^qu(x)\ddd\gamma(x)\r)^{1/q}\leq C \l(\int_{\bR^d}|f(x)|^p v(x)\ddd\gamma(x)\r)^{1/p}.\]
\end{thm}
\begin{thm}\label{GM_2001_thm6.5}
Given $a\in(0,\infty)$, $1< p\leq q<\infty$, $0<\alpha<d$ and a pair of weights $(u,v)$, if 
\[(u,v)\in\bigcup_{6a+9\sqrt{d}a^2<b'}\bigcup_{0<\beta'<\alpha} \mathscr{M}_{p,q,\beta'}^{b'},\]
then $\tilde{I}_{\alpha}^a$ is bounded from $L^p(\bR^d, v, \gamma)$ to $L^{q}(\bR^d, u, \gamma)$, that is, there exists a constant $C>0$ such that
\[\l(\int_{\bR^d}\l(\tilde{I}_{\alpha}^a(f)(x)\r)^q u(x)\ddd\gamma(x)\r)^{1/q}\leq C\l(\int_{\bR^d} |f(x)|^p v(x)\ddd\gamma(x)\r)^{1/p}.\]
\end{thm}

The highlights of the paper are as follows. In Section \ref{DHY_2101_preliminaries}, we give some basic facts used in the proofs of the desired main results, as well as the radial versions of $M_{\alpha}^a$ and $A_{p,q,\alpha}^a$. In Section \ref{DHY_2101_TLMTW}, we investigate some properties for the local Muckenhoupt type weights and prove the two-weight weak-type boundedness of $M_{\alpha}^a$ and $\tilde{I}_{\alpha}^a$. Meanwhile a natural question (see Question \ref{DHY_2101_question1} below) arises here. In Section \ref{DHY_2101_TLSTW}, we study the similar properties for the local Sawyer type weights and prove the two-weight strong-type boundedness of $M_{\alpha}^a$ and $\tilde{I}_{\alpha}^a$. Furthermore, a similar question (see Question \ref{DHY_2101_question2} below) is stated at the end of Section \ref{DHY_2101_TLSTW}.

We end this section with some notions and notations. Hereafter, we will be working in $\bR^d$ and $d$ will always denote the dimension. We will denote by $C$ or its variants a positive constant independent of the main involved parameters, and use $f\lesssim g$ to denote $f\leq C g$; particularly, if $f\lesssim g\lesssim f$, then we will write $f\sim g$. If necessary, we will denote the dependence of the constants parenthetically, e.g., $C=C(a,d)$ or $C=C_{a,d}$. Similarly, $f\lesssim_{a,d} g$ will denote $f\leq C_{a,d} g$ and $f\sim_{a,d}g$ will denote $C_{a,d} f\leq g\leq C_{a,d} f$. By a weight we will always mean a locally integrable function and non-negative almost everywhere on $\bR^d$ (with respect to the associated measure $\gamma$). For a given weight $\omega$, the weighted Gaussian Lebesgue norms on $\bR^d$ will be denoted by
\[\|f\|_{L^p(\bR^d, \omega, \gamma)}:=\l(\int_{\bR^d} |f(x)|^p\omega(x)\ddd\gamma(x)\r)^{\frac{1}{p}},\]
 and
  \[\|f\|_{L^{p,\infty}(\bR^d, \omega, \gamma)}:=\sup_{\lambda>0}\frac{1}{\lambda}\l(\int_{\{x\in\bR^d: f(x)>\lambda\}} \omega(x)\ddd\gamma(x)\r)^{\frac{1}{p}}.\]

\section{Preliminaries}\label{DHY_2101_preliminaries}
It is easy to see that the probability measure $\gamma$ is highly concentrated around the origin with exponential decay at infinity. Thereby it is not a doubling measure, i.e., there is no constant $C>0$, independent of $x\in\bR^d$ and $r>0$, such that
\[\gamma\l(B(x,2r)\r) \leq C \gamma\l(B(x,x)\r)\]
holds for all $x\in\bR^d$ and $r>0$. Here $\gamma(B):=\int_B \ddd\gamma(x)$. See~\cite[Appendix 10.3]{GHA_2019} for more details. Hence we know that the Gaussian measure space is not a homogeneous type space in the sense of Coifman and Weiss~\cite{CW_1977}. However, as we have mentioned in Section \ref{DHY_2101_introduction}, if we define the family of admissible cubes $\mathscr{Q}_a$, then \cite[Proposition 2.1]{MM_2007} points out that
\begin{equation}\label{DHY_2101_equivalent}
e^{|c_Q|^2}\sim_{a}e^{|x|^2}
\end{equation}
holds for all $Q\in \mathscr{Q}_a$ and all $x\in Q$. From this estimate we conclude that
\begin{equation}\label{DHY_2101_equivalent_2}
\gamma(Q)=\int_Q \ddd\gamma(x)\sim_{a,d} \;e^{-|c_Q|^2}\int_Q \ddd x =e^{-|c_Q|^2}\ell(Q)^d
\end{equation}
holds for all $Q\in \mathscr{Q}_a$. Furthermore the Gaussian measure is doubling and reverse doubling if we restrict it to $\mathscr{Q}_a$. In other words, there exist constants $C_1=C_{a,d}\geq 1$ and $C_2=C'_{a,d}>1$ such that for all $Q\in \mathscr{Q}_a$ we have
\[\gamma(2Q)\leq C_1 \gamma(Q), \quad \gamma(2Q)\geq C_2\gamma(Q).\]
Hence we say $\gamma$ satisfies the locally doubling condition and locally reverse doubling condition on $\mathscr{Q}_a$.

On the other hand, the Gaussian measure is trivially a $d$-dimensional measure in $\bR^d$, i.e.,
\[\gamma(Q)\leq \ell(Q)^d\]
holds for all cubes $Q$ in $\bR^d$. Therefore some results on the $d$-dimensional measure, such as~\cite{Tolsa_2001} and~\cite{GM_2001}, may be useful in Gaussian harmonic analysis.

By using the estimates (\ref{DHY_2101_equivalent}) and (\ref{DHY_2101_equivalent_2}), we conclude that
\begin{align*}
\frac{1}{[\gamma(Q)]^{1-\dot{\alpha}}}\int_Q |f(y)|\ddd \gamma(y)
&\sim_{a,d} e^{-|c_Q|^2}\frac{1}{[\gamma(Q)]^{1-\dot{\alpha}}}\int_Q |f(y)|\ddd y \\
&\sim_{a,d} e^{-\dot{\alpha}|c_Q|^2}\frac{1}{\ell(Q)^{d-\alpha}}\int_Q |f(y)|\ddd y \\
&\sim_{a,d} \frac{1}{\ell(Q)^{d-\alpha}} \int_Q |f(y)| e^{-\dot{\alpha}|y|^2} \ddd y
\end{align*}
holds for all $Q\in\mathscr{Q}_a$. Now we can give the pointwise equivalent radial version of $M_{\alpha}^a$ as follows.
\begin{definition}
Let $a\in(0,\infty)$, $\dot{\alpha}\in [0,1)$, $\alpha=\dot{\alpha}d$ and $f\in L_{loc}^{1}(\gamma)$. We define the \textit{local fractional maximal operator $M^{a}_{\alpha}$ on Gaussian measure spaces} by setting
\begin{align*}
M^{a}_{\alpha}(f)(x)&:=\sup_{Q\in\mathscr{Q}_a(x)}\frac{1}{[\gamma(Q)]^{1-\dot{\alpha}}}\int_Q |f(y)|\ddd \gamma(y) \\
&:\sim_{a,d} \sup_{Q\in\mathscr{Q}_a(x)} \frac{1}{\ell(Q)^{d-\alpha}} \int_Q |f(y)| \ddd \gamma'(y),
\end{align*}
where $\ddd \gamma'(y)=e^{-\dot{\alpha}|y|^2}\ddd y=e^{-\alpha|y|^2/d}\ddd y$.
\end{definition}

To use the dyadic analysis on Gaussian measure spaces more conveniently and adapt the $A_{p,q,\alpha}^a$ condition to our radialization method, we introduce the following $\mathscr{A}_{p,q,\alpha}^a$ condition.
\begin{definition}
Given $0<a<\infty$, $0\leq\dot{\alpha}<1$, $\alpha=\dot{\alpha}d$ and $1< p\leq q<\infty$, we say that a pair of weights $(u,v)\in\mathscr{A}_{p,q,\alpha}^a$ if
\begin{align*}
[u,v]_{\mathscr{A}_{p,q,\alpha}^a}=\sup_{Q\in\mathscr{Q}_a} \frac{1}{\ell(Q)^{d-\alpha}}\l(\int_Q  u(x) \ddd  \gamma'(x)\r)^{1/q}\l(\int_Q v(x)^{1-p'}\ddd\gamma'(x)\r)^{1/{p'}}<\infty.
\end{align*}
In the case $p=1$, we say $(u,v)\in \mathscr{A}_{1,q,\alpha}^a$ if there exists a constant $C$ such that for every cube $Q\in\mathscr{Q}_a$ the inequality
\[\frac{1}{\ell(Q)^{d-\alpha}} \l(\int_Q  u(x) \ddd  \gamma'(x)\r)^{1/q} \leq Cv(x)\]
holds for $\gamma'$-a.e. $x\in Q$.
\end{definition}

We can show the following close relation between the $A_{p,q,\alpha}^a$ condition and the $\mathscr{A}_{p,q,\alpha}^a$ condition.
\begin{prop}\label{DHY_RBAA}
Let
\[u'(x):=u(x)e^{-(1-\dot{\alpha})|x|^2},\quad v'(x):=v(x)e^{-(1-\dot{\alpha})|x|^2}.\]
Then
\[(u,v)\in A_{p,q,\alpha}^a\Leftrightarrow (u',v')\in \mathscr{A}_{p,q,\alpha}^a.\]
\end{prop}

\begin{proof}
We focus on proving the case $p>1$ since the case $p=1$ is essentially the same. For every fixed $Q\in\mathscr{Q}_a$, we have known that
\begin{equation}\label{DHY_LDP}
e^{-|x|^2}\sim_a e^{|y|^2}
\end{equation}
holds for any $x\in Q$ and any $y\in Q$. Hence we conclude
\begin{align*}
&\gamma(Q)^{(\dot{\alpha}+1/q-1/p)}\l(\frac{1}{\gamma(Q)} \int_Q u(x) \ddd \gamma(x)\r)^{1/q}\l(\frac{1}{\gamma(Q)} \int_Q v(x)^{1-p'}\ddd\gamma(x)\r)^{1/{p'}} \\
&\sim_{a,d} \ell(Q)^{\alpha+d/q-d/p} e^{-|c_Q|^2(\dot{\alpha}+1/q-1/p)} \l(\frac{e^{-|c_Q|^2}}{\ell(Q)^d e^{-|c_Q|^2}} \int_Q u(x) \ddd x\r)^{1/q}\l(\frac{e^{-|c_Q|^2}}{\ell(Q)^d e^{-|c_Q|^2}} \int_Q v(x)^{1-p'}\ddd x\r)^{1/{p'}} \\
&=\ell(Q)^{\alpha+d/q-d/p} e^{-|c_Q|^2(\dot{\alpha}+1/q-1/p)} \l(\frac{1}{\ell(Q)^d}\int_Q u(x)\ddd x\r)^{1/q} \l(\frac{1}{\ell(Q)^d}\int_Q v(x)^{1-p'}\ddd x\r)^{1/{p'}}.
\end{align*}
On the other hand, using the estimate (\ref{DHY_LDP}) again we have
\begin{align*}
&\frac{1}{\ell(Q)^{d-\alpha}}\l(\int_Q u(x)e^{-(1-\dot{\alpha})|x|^2}\ddd\gamma'(x)\r)^{1/q}\l(\int_Q v(x)^{1-p'}e^{-(1-\dot{\alpha})|x|^2(1-p')}\ddd\gamma'(x)\r)^{1/{p'}} \\
&=\frac{1}{\ell(Q)^{d-\alpha}} \l(\int_Q u(x)\ddd x\r)^{1/q} e^{-|c_Q|^2/q}\l(\int_Q v(x)^{1-p'}\ddd x\r)^{1/{p'}} e^{-|c_Q|^2(\dot{\alpha}-1/p)} \\
&=\ell(Q)^{\alpha+d/q-d/p} e^{-|c_Q|^2(\dot{\alpha}+1/q-1/p)} \l(\frac{1}{\ell(Q)^d}\int_Q u(x)\ddd x\r)^{1/q} \l(\frac{1}{\ell(Q)^d}\int_Q v(x)^{1-p'}\ddd x\r)^{1/{p'}}.
\end{align*}
These two facts yield the desired result.
\end{proof}

\section{The local Muckenhoupt type weights}\label{DHY_2101_TLMTW}
Due to the Proposition \ref{DHY_RBAA} above, we investigate some properties of the $\mathscr{A}_{p,q,\alpha}^a$ condition. In the one-weight case , Wang~et~al.~\cite{WZL_2016} point out that the local Moukenhoupt weights $A_{p}^a$ on Gaussian measure spaces have the property $A_{p}^a=A_{p}^b$ and then $A_{p,q}^a=A_{p,q}^b$; but in the two-weight case, the similar result $\mathscr{A}_{p,q,\alpha}^a=\mathscr{A}_{p,q,\alpha}^b$ is not always true. To see this fact, we need the following example first.
\begin{example}\label{WZL_2016_prop3.2}
Let $0<a<b<\infty$, $n\in \bZ_{+}$, $u(x)$ and $v(x)$ be even functions on $\bR^1$. When $x\in\bR^{+}$,
\[u(x)=\l\{\begin{array}{ll}
1, & x\in (0,1), \\
n^q, & x\in \l(1+(n-1)a+(n-1)b, 1+(n-1)a+(n-1)b+\frac{b-a}{2}\r), \\
1/{n^q}, & x\in \l(1+(n-1)a+(n-1)b+\frac{b-a}{2}, 1+na+nb\r),
\end{array}\r.\]
with $n\geq1$ and
\[v(x)=\l\{\begin{array}{ll}
1, & x\in(0,1+b), \\
n^p, & x\in \l(1+(n-2)a+(n-1)b, 1+(n-1)a+nb-\frac{b-a}{2}\r),\\
1/{n^p}, & x\in \l(1+(n-1)a+nb-\frac{b-a}{2}, 1+(n-1)a+nb\r),
\end{array}\r.\]
with $n\geq2$. Define
\[[u,v]_{\mathfrak{A}_{p,q,\alpha}^a}=\sup_{Q\in \mathscr{Q}'_a}\frac{1}{\ell(Q)^{1-\alpha}}\l(\int_Q u(x)\ddd x\r)^{1/q} \l(\int_Q v(x)^{1-p'} \ddd x\r)^{1/{p'}},\]
where $\mathscr{Q}'_a=\{Q\subset \bR^d: \ell(Q)\leq a\}$.
If $1/q-1/p+\alpha\geq0$, then
\[[u,v]_{\mathfrak{A}_{p,q,\alpha}^a}<\infty, \quad [u,v]_{\mathfrak{A}_{p,q,\alpha}^b}=\infty.\]
\end{example}

\begin{proof}
We only need to consider all the cubes $Q\in\mathscr{Q}'_{a}$ on $\bR^{+}$ due to the symmetry. Setting $\sigma(x)=v(x)^{1-p'}$, we can draw the function graphs of $u(x)^{1/q}$ and $\sigma(x)^{1/p}$ on $\bR^{+}$.

\begin{figure}[htbp]\centering
\definecolor{zzttff}{rgb}{0.6,0.2,1}
\definecolor{ttffqq}{rgb}{0.2,1,0}
\definecolor{ffffff}{rgb}{1,1,1}
\definecolor{qqffqq}{rgb}{0,1,0}
\definecolor{ffqqqq}{rgb}{1,0,0}
\definecolor{xdxdff}{rgb}{0.49019607843137253,0.49019607843137253,1}
\begin{tikzpicture}[line cap=round,line join=round,>=triangle 45,x=1cm,y=1cm]
\begin{axis}[
x=1cm,y=1cm,
axis lines=middle,
ymajorgrids=true,
xmajorgrids=false,
xmin=-0.38872411953161656,
xmax=10.88528457209615,
ymin=-0.3899465493386252,
ymax=4.817601473758484,
xtick={0,1},
ytick={0,1/2,1,2,3,4},]
\clip(-0.38872411953161656,-1.1799465493386252) rectangle (10.88528457209615,4.817601473758484);
\draw [line width=2pt,color=ffqqqq] (1,0)-- (3,0);
\draw [line width=2pt,color=ffqqqq] (4,0)-- (6,0);
\draw [line width=2pt,color=ffqqqq] (7,0)-- (9,0);
\draw [line width=2.4pt] (1.5,0)-- (2.5,0);
\draw [line width=2.4pt] (4.5,0)-- (5.5,0);
\draw [line width=2.4pt] (7.5,0)-- (8.5,0);
\draw [line width=1.2pt,color=qqffqq] (0,1)-- (3,1);
\draw [line width=1.2pt,color=qqffqq] (3,0.5)-- (5.5,0.5);
\draw [line width=1.2pt,color=qqffqq] (6,0.33333)-- (8.5,0.3333);
\draw [line width=1.2pt,color=qqffqq] (5.5,2)-- (6,2);
\draw [line width=1.2pt,color=qqffqq] (8.5,3)-- (9,3);
\draw [line width=1.2pt,color=qqffqq] (9,0.25)-- (11.5,0.25);
\draw [line width=1.2pt,dotted,color=zzttff] (0,1)-- (4,1);
\draw [line width=1.2pt,dotted,color=zzttff] (4.5,0.5)-- (7,0.5);
\draw [line width=1.2pt,dotted,color=zzttff] (7.5,0.3333)-- (11.5,0.33333);
\draw [line width=1.2pt,dotted,color=zzttff] (7,3)-- (7.5,3);
\draw [line width=1.2pt,dotted,color=zzttff] (4,2)-- (4.5,2);
\begin{scriptsize}
\draw [fill=xdxdff] (0,0) circle (1.5pt);
\draw (0.1862061345801013,-.1957683491553985) node {$O$};
\draw [fill=ffqqqq] (1,0) circle (1.5pt);
\draw[color=ffqqqq] (1.0862061345801013,0.1557683491553985) node {$B_1$};
\draw [fill=ffqqqq] (3,0) circle (1.5pt);
\draw[color=ffqqqq] (3.1104726568048573,-0.1957683491553985) node {$B'_1$};
\draw [fill=ffqqqq] (4,0) circle (1.5pt);
\draw[color=ffqqqq] (4.081217568949814,-0.1957683491553985) node {$B_2$};
\draw [fill=ffqqqq] (6,0) circle (1.5pt);
\draw[color=ffqqqq] (6.113009245532282,-0.1957683491553985) node {$B'_{2}$};
\draw [fill=ffqqqq] (7,0) circle (1.5pt);
\draw[color=ffqqqq] (7.08375415767724,-0.1957683491553985) node {$B_3$};
\draw [fill=ffqqqq] (9,0) circle (1.5pt);
\draw[color=ffqqqq] (9.115545834259708,-0.1957683491553985) node {$B'_3$};
\draw [fill=black] (1.5,0) circle (1.5pt);
\draw[color=black] (1.4828663221891492,-0.1957683491553985) node {$A_1$};
\draw [fill=black] (2.5,0) circle (1.5pt);
\draw[color=black] (2.5138124691958096,-0.1957683491553985) node {$A'_1$};
\draw [fill=black] (4.5,0) circle (1.5pt);
\draw[color=black] (4.485402910916574,-0.1957683491553985) node {$A_2$};
\draw [fill=black] (5.5,0) circle (1.5pt);
\draw[color=black] (5.516349057923235,-0.1957683491553985) node {$A'_2$};
\draw [fill=black] (7.5,0) circle (1.5pt);
\draw[color=black] (7.480414345286288,-0.1957683491553985) node {$A_3$};
\draw [fill=black] (8.5,0) circle (1.5pt);
\draw[color=black] (8.511360492292947,-0.1957683491553985) node {$A'_3$};
\draw [fill=qqffqq] (0,1) circle (1pt);
\draw [fill=ffffff] (3,1) circle (1pt);
\draw [fill=qqffqq] (3,0.5) circle (1pt);
\draw [fill=qqffqq] (5.5,0.5) circle (1pt);
\draw [fill=ttffqq] (6,0.33333) circle (1pt);
\draw [fill=qqffqq] (8.5,0.3333) circle (1pt);
\draw [fill=qqffqq] (9,0.25) circle (1pt);
\draw [fill=ffffff] (5.5,2) circle (1pt);
\draw [fill=ffffff] (4,1) circle (1pt);
\draw [fill=ffffff] (4.5,0.5) circle (1pt);
\draw [fill=ffffff] (7,0.5) circle (1pt);
\draw [fill=ffffff] (7.5,0.33333) circle (1pt);
\draw [fill=ffffff] (6,2) circle (1pt);
\draw [fill=ffffff] (8.5,3) circle (1pt);
\draw [fill=ffffff] (9,3) circle (1pt);
\draw [fill=ffffff] (1,1) circle (1pt);
\draw [fill=ffffff] (1.5,1) circle (1pt);
\draw [fill=ffffff] (2.5,1) circle (1pt);
\draw [fill=ffffff] (4,2) circle (1pt);
\draw [fill=ffffff] (6,0.5) circle (1pt);
\draw [fill=ffffff] (9,0.33333) circle (1pt);
\draw [fill=ffffff] (7,3) circle (1pt);
\draw [fill=ffffff] (7.5,3) circle (1pt);
\draw [fill=ffffff] (4.5,2) circle (1pt);
\end{scriptsize}
\end{axis}
\end{tikzpicture}
\caption{The fuctions $u(x)^{1/q}$ and $\sigma(x)^{1/p}$ on $\bR^+$}
\label{figure_1}
\end{figure}
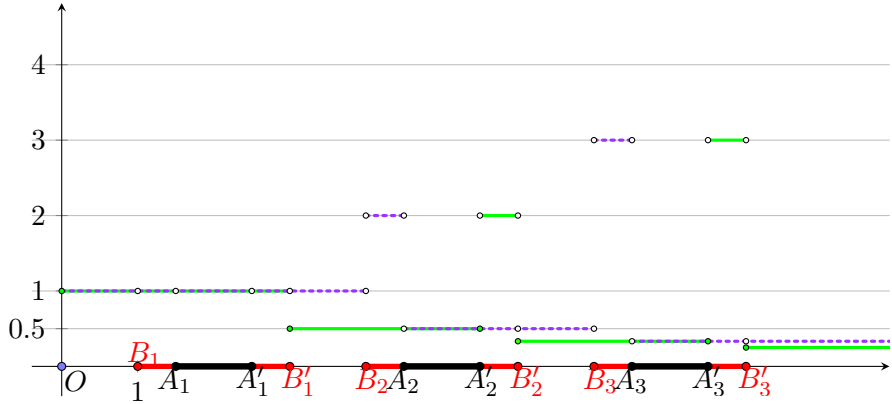

As shown in \autoref{figure_1}, the green full line segments are the function graph of $\sigma(x)^{1/p}$ and the purple dotted line segments are the function graph of $u(x)^{1/q}$. The lengths of these intervals satisfy the following
\[(B_i,B'_i)=b,\quad(A_i,A'_i)=a,\quad(B_i,A_i)=(A'_i,B'_i)=\frac{b-a}{2}, \quad (B'_i,B_{i+1})=a\]
for every $i\in\bZ_{+}$. Based on this \autoref{figure_1}, we shall prove the desired results more intuitively. Set
\[Q_{n_1}=(B_n,A_n)=\l(1+(n-1)a+(n-1)b, 1+(n-1)a+(n-1)b+\frac{b-a}{2}\r),\]
\[Q_{n_2}=(A'_n,B'_n)=\l(1+(n-1)a+nb-\frac{b-a}{2}, 1+(n-1)a+nb\r)\]
and
\[Q_n=(B_n,B'_n)=\Big(1+(n-1)a+(n-1)b, 1+(n-1)a+nb\Big).\]
For fixed $Q\in\mathscr{Q}'_a$, since $(A_i,A'_i)=(B'_i,B_{i+1})=a$ and $\ell(Q)\leq a$, the open interval $Q$ cannot intersect with more than one of these intervals $Q_{n_1}$ and $Q_{n_2}$. Therefore, when $Q\cap Q_{n_1}\neq\emptyset$ we obtain
\begin{align*}
\frac{1}{\ell(Q)^{1-\alpha}}\l(\int_Q u(x)\ddd x\r)^{1/q} \l(\int_Q v(x)^{1-p'} \ddd x\r)^{1/{p'}}
&\leq \frac{1}{\ell(Q)^{1-\alpha}}\l(\int_Q n^q\ddd x\r)^{1/q} \l(\int_Q \frac{1}{n^{p'}} \ddd x\r)^{1/{p'}} \\
&= \frac{1}{\ell(Q)^{1-\alpha}} \ell(Q)^{1/q+1/{p'}} \\
&= \ell(Q)^{1/q-1/p+\alpha} \leq a^{1/q-1/p+\alpha};
\end{align*}
and when $Q\cap Q_{n_2}\neq\emptyset$ we also obtain
\begin{align*}
\frac{1}{\ell(Q)^{1-\alpha}}\l(\int_Q u(x)\ddd x\r)^{1/q} \l(\int_Q v(x)^{1-p'} \ddd x\r)^{1/{p'}}
&\leq \frac{1}{\ell(Q)^{1-\alpha}}\l(\int_Q \frac{1}{n^q} \ddd x\r)^{1/q} \l(\int_Q n^{p'} \ddd x\r)^{1/{p'}} \\
&= \frac{1}{\ell(Q)^{1-\alpha}} \ell(Q)^{1/q+1/{p'}} \\
&= \ell(Q)^{1/q-1/p+\alpha} \leq a^{1/q-1/p+\alpha},
\end{align*}
where we have used the assumption that $1/q-1/p+\alpha\geq 0$. Hence we have proved the result
\[[u,v]_{\mathfrak{A}_{p,q,\alpha}^a}\leq a^{1/q-1/p+\alpha}<\infty.\]
On the other hand, by chosing $Q=Q_n\in \mathscr{Q}'_b$, we conclude that
\begin{align*}
[u,v]_{\mathfrak{A}_{p,q,\alpha}^b}&=\sup_{Q\in \mathscr{Q}'_b}\frac{1}{\ell(Q)^{d-\alpha}}\l(\int_Q u(x)\ddd x\r)^{1/q} \l(\int_Q v(x)^{1-p'} \ddd x\r)^{1/{p'}} \\
&\geq \frac{1}{\ell(Q_n)^{1-\alpha}}\l(\int_{Q_{n_1}} u(x)\ddd x\r)^{1/q} \l(\int_{Q_{n_2}} v(x)^{1-p'} \ddd x\r)^{1/{p'}} \\
&=\frac{1}{\ell(Q_n)^{1-\alpha}} \cdot n^2 \cdot \ell(Q_{n_1})^{d/q} \cdot \ell(Q_{n_2})^{d/{p'}} \\
&=n^2\frac{\l(\frac{b-a}{2}\r)^{1/q+1/{p'}}}{b^{d-\alpha}}.
\end{align*}
Letting $n\to\infty$, we get the desired result $[u,v]_{\mathfrak{A}_{p,q,\alpha}^b}=\infty$.  This completes the proof.
\end{proof}

\begin{prop}\label{WZL_2016_prop3.2_2}
Let $0<a<b<\infty$, $0\leq\alpha<d$ and $1<p\leq q<\infty$. Then
\[\mathscr{A}_{p,q,\alpha}^b\subsetneqq\mathscr{A}_{p,q,\alpha}^a.\]
\end{prop}

\begin{proof}
It's trivial that $\mathscr{A}_{p,q,\alpha}^b\subset\mathscr{A}_{p,q,\alpha}^a$ and by Remark \ref{DHY_1} together with Propsition \ref{DHY_RBAA} we only need to concentrate on the situation $1/q-1/p+\alpha\geq0$. In other words, it is enough to find $(u,v)$ belongs to $\mathscr{A}_{p,q,\alpha}^a$ but not in $\mathscr{A}_{p,q,\alpha}^b$ under the assumption $1/q-1/p+\alpha\geq0$. Inspired by Example \ref{WZL_2016_prop3.2}, on $\bR^{+}$, we set
\[u(x)=\l\{\begin{array}{ll}
e^{\dot{\alpha}|x|^2}, & x\in (0,1), \\
n^{kq}e^{\dot{\alpha}|x|^2}, & x\in \l(1+(n-1)a+\sum_{i=1}^{n-1}b_i, 1+(n-1)a+\sum_{i=1}^{n-1}b_i+\frac{b_n-a_n}{2}\r), \\
n^{-kq}e^{\dot{\alpha}|x|^2}, & x\in \l(1+(n-1)a+\sum_{i=1}^{n-1}b_i+\frac{b_n-a_n}{2}, 1+na+\sum_{i=1}^{n}b_i\r),
\end{array}\r.\]
with $n\geq1$ and
\[v(x)=\l\{\begin{array}{ll}
e^{\frac{\dot{\alpha}|x|^2}{1-p'}}, & x\in(0,1), \\
n^{kp}e^{\frac{\dot{\alpha}|x|^2}{1-p'}}, & x\in \l(1+(n-2)a+\sum_{i=1}^{n-1}b_i, 1+(n-1)a+\sum_{i=1}^{n}b_i-\frac{b_n-a_n}{2}\r), \\
n^{-kp}e^{\frac{\dot{\alpha}|x|^2}{1-p'}}, & x\in \l(1+(n-1)a+\sum_{i=1}^{n}b_i-\frac{b_n-a_n}{2}, 1+(n-1)a+\sum_{i=1}^{n}b_i\r),
\end{array}\r.\]
with $n\geq2$ where $a_n, b_n$ satisfy the condition
\begin{equation}\label{condition_1}
\l\{\begin{array}{l}
x_1-\frac{b}{2x_1}=1 \\
b_n=\frac{b}{x_n}>0 \\
1+(n-1)a+\sum_{i=1}^{n-1}b_i=x_n-\frac{b}{2x_n} \\
a_n=\frac{a}{x_n}
\end{array}\r.,
\end{equation}
and the parameter $k$ will be chosen later. The exponent terms in $u(x)$, $v(x)$ are used to offset the measure $\ddd\gamma'(x)$ to Lebesgue measure $\ddd x$ and the condition (\ref{condition_1}) can be shown in \autoref{figure_2}.

\begin{figure}[htbp]\centering
\definecolor{ffqqqq}{rgb}{1,0,0}
\definecolor{xdxdff}{rgb}{0.49019607843137253,0.49019607843137253,1}
\begin{tikzpicture}[line cap=round,line join=round,>=triangle 45,x=1cm,y=1cm]
\begin{axis}[
x=1cm,y=1cm,
axis lines=middle,
xmin=-0.3,
xmax=14.6,
ymin=-1,
ymax=1,
xtick={0,1},
ytick={0},]
\clip(-0.3062670637150298,-4.222722017611691) rectangle (14.653039395116444,4.562132588553177);
\draw [line width=2pt,color=ffqqqq] (1,0)-- (4,0);
\draw [line width=2pt,color=ffqqqq] (7,0)-- (9.4,0);
\draw [line width=2.4pt] (1.6,0)-- (3.4,0);
\draw [line width=2.4pt] (7.4,0)-- (9,0);
\draw [line width=2pt,color=ffqqqq] (12.4,0)-- (14.2,0);
\draw [line width=2.4pt] (12.6,0)-- (14,0);
\begin{scriptsize}
\draw [fill=xdxdff] (0,0) circle (2.5pt);
\draw[color=xdxdff] (0.09053941786831858,0.23583969906796914) node {$O$};
\draw [fill=ffqqqq] (1,0) circle (1.5pt);
\draw[color=ffqqqq] (1.126645230891506,0.22481729680176477) node {$B_1$};
\draw [fill=ffqqqq] (4,0) circle (1.5pt);
\draw[color=ffqqqq] (4.168828256363844,0.22481729680176477) node {$B'_1$};
\draw [fill=ffqqqq] (7,0) circle (1.5pt);
\draw[color=ffqqqq] (7.116754851257398,0.22481729680176477) node {$B_2$};
\draw [fill=ffqqqq] (9.4,0) circle (1.5pt);
\draw[color=ffqqqq] (9.574750556620917,0.22481729680176477) node {$B'_2$};
\draw [fill=black] (1.6,0) circle (1.5pt);
\draw[color=black] (1.6998101487341204,0.22481729680176477) node {$A_1$};
\draw [fill=black] (3.4,0) circle (1.5pt);
\draw[color=black] (3.551573729456413,0.22481729680176477) node {$A'_{1}$};
\draw [fill=ffqqqq] (2.5,0) circle (2.5pt);
\draw[color=ffqqqq] (2.6256919390952667,0.26890690586658217) node {$x_1$};
\draw [fill=black] (7.4,0) circle (1.5pt);
\draw[color=black] (7.52458373510695,0.22481729680176477) node {$A_2$};
\draw [fill=black] (9,0) circle (1.5pt);
\draw[color=black] (9.1,0.22481729680176477) node {$A'_2$};
\draw [fill=ffqqqq] (8.2,0) circle (2.5pt);
\draw[color=ffqqqq] (8.318196698273647,0.26890690586658217) node {$x_2$};
\draw [fill=ffqqqq] (12.4,0) circle (1.5pt);
\draw[color=ffqqqq] (12.35,0.22481729680176477) node {$B_3$};
\draw [fill=ffqqqq] (14.2,0) circle (1.5pt);
\draw[color=ffqqqq] (14.36341832997056,0.22481729680176477) node {$B'_3$};
\draw [fill=black] (12.6,0) circle (1.5pt);
\draw[color=black] (12.721080392306146,0.22481729680176477) node {$A_3$};
\draw [fill=black] (14,0) circle (1.5pt);
\draw[color=black] (13.95,0.22481729680176477) node {$A'_3$};
\draw [fill=ffqqqq] (13.3,0) circle (2.5pt);
\draw[color=ffqqqq] (13.42651413734321,0.26890690586658217) node {$x_3$};
\end{scriptsize}
\end{axis}
\end{tikzpicture}
\caption{The admissible partition of $\bR^+$}
\label{figure_2}
\end{figure}
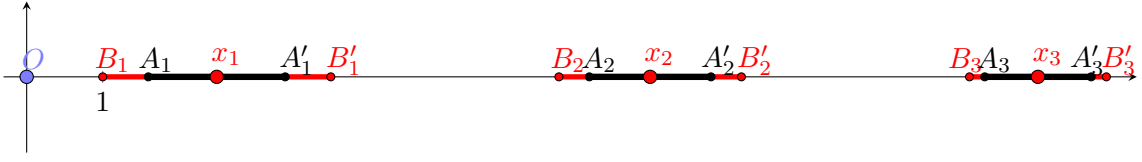

As you can see in \autoref{figure_2}, the red intervals $(B_n,B'_n)$ and the black intervals $(A_n,A'_n)$ have the same centers $x_n$. Furthermore, the lengths of these intervals satisfy the following
\[B'_n-B_n=b/{x_n}=b_n,\quad A'_n-A_n=a/{x_n}=a_n, \quad B_{n+1}-B'_n=a\]
so that $(B_n,B'_n)\in\mathscr{Q}_b$, $(A_n,A'_n)\in\mathscr{Q}_a$ and these intervals $(B_n,B'_n)$ are ``far enough to each other''.

By even extension, we get $u(x)$ and $v(x)$ on $\bR^1$. Similarly set
\[Q_{n_1}=(B_n,A_n)=\l(1+(n-1)a+\sum_{i=1}^{n-1}b_i, 1+(n-1)a+\sum_{i=1}^{n-1}b_i+\frac{b_n-a_n}{2}\r),\]
\[Q_{n_2}=(A'_n,B'_n)=\l(1+(n-1)a+\sum_{i=1}^{n}b_i-\frac{b_n-a_n}{2}, 1+(n-1)a+\sum_{i=1}^{n}b_i\r)\]
and
\[Q_n=(B_n,B'_n)=\l(1+(n-1)a+\sum_{i=1}^{n-1}b_i, 1+(n-1)a+\sum_{i=1}^{n}b_i\r).\]
For fixed $Q\in\mathscr{Q}_a$, we should avoid the situation that $Q\cap Q_{n_1}\neq \emptyset$ and $Q\cap Q_{n_2}\neq \emptyset$ hold simultaneously. Actually, this situation may appear if $n=1$. However, we claim that the open interval $Q\in\mathscr{Q}_a$ can not intersect with more than one of these intervals $Q_{n_1}$ and $Q_{n_2}$ if $n\geq2$.

Based on this claim and the assumption $1/q-1/p+\alpha\geq0$, we can follow the scheme of the proof in Example \ref{WZL_2016_prop3.2}. For every fixed $Q\in\mathscr{Q}_a$, if $Q\cap Q_{n}\neq\emptyset$ with $n\geq2$, we conclude that
\begin{align*}
\frac{1}{\ell(Q)^{1-\alpha}}\l(\int_Q  u(x) \ddd  \gamma'(x)\r)^{1/q}\l(\int_Q v(x)^{1-p'}\ddd\gamma'(x)\r)^{1/{p'}} 
&\leq \frac{1}{\ell(Q)^{1-\alpha}}\cdot n^k\cdot n^{-k} \cdot |Q|^{1/q} \cdot |Q|^{1/{p'}} \\
&=\ell(Q)^{1/q-1/p+\alpha} \leq a^{1/q-1/p+\alpha},
\end{align*}
where $|Q|$ is the Lebesgue measure of $Q$; if $Q\cap Q_1\neq \emptyset$, we also conclude that
\begin{align*}
&\frac{1}{\ell(Q)^{1-\alpha}}\l(\int_Q  u(x) \ddd  \gamma'(x)\r)^{1/q}\l(\int_Q v(x)^{1-p'}\ddd\gamma'(x)\r)^{1/{p'}} \\
&\leq \frac{1}{\ell(Q)^{1-\alpha}}\cdot |Q|^{1/q} \cdot |Q|^{1/{p'}} \\
&=\ell(Q)^{1/q-1/p+\alpha} \leq a^{1/q-1/p+\alpha}.
\end{align*}
Hence we can obtain the result
\[[u,v]_{\mathscr{A}_{p,q,\alpha}^a}\leq a^{1/q-1/p+\alpha}<\infty.\]
However by chosing $Q=Q_n\in \mathscr{Q}'_b$, we can similarly get
\begin{align*}
[u,v]_{\mathscr{A}_{p,q,\alpha}^b}&\geq \frac{1}{\ell(Q_n)^{1-\alpha}}\l(\int_{Q_{n_1}}  u(x) \ddd  \gamma'(x)\r)^{1/q}\l(\int_{Q_{n_2}} v(x)^{1-p'}\ddd\gamma'(x)\r)^{1/{p'}} \\
&= n^{2k} \cdot \frac{1}{\ell(Q_n)^{1-\alpha}} \cdot |Q_{n_1}|^{1/q} \cdot |Q_{n_2}|^{1/{p'}} \\
&=n^{2k} \cdot \frac{1}{b_n^{1-\alpha}} \cdot \l(\frac{b_n-a_n}{2}\r)^{1/q+1/{p'}} \\
&=\frac{\l(\frac{b-a}{2}\r)^{1/q+1/{p'}}}{b^{1-\alpha}} n^{2k} \cdot |x_n|^{-\alpha-1/q+1/p} \\
&\geq \frac{\l(\frac{b-a}{2}\r)^{1/q+1/{p'}}}{b^{1-\alpha}} n^{2k} \cdot (1+na+nb)^{-\alpha-1/q+1/p},
\end{align*}
where we have used the facts that $\alpha+1/q-1/p\geq0$ and
\[|x_n|\leq 1+na+\sum_{i=1}^n b_i\leq 1+na+nb.\]
Finally by taking $k=\alpha+1/q-1/p+1>0$, we get the desired weights $u(x)$ and $v(x)$.

It remains to prove the claim above. It's easy to see that if $n\geq 2$, then $x_{n}>1+a>\sqrt{a/2}$. Notice that when $x>\sqrt{a/2}$, the function $f(x)=x+\frac{a}{2x}$ is increasing. Hence if $n\geq2$, then for all $0<c_Q\leq1$, we have
\[x_n+\frac{a}{2x_n}>a+1\geq c_Q+a;\]
and for all $1<c_Q<x_n$, we have
\[x_n+\frac{a}{2x_n}>c_Q+\frac{a}{2c_Q}.\]
On the other hand, it is obvious that if $c_Q\geq x_n$, then
\[c_Q-\frac{a}{2c_Q}\geq x_n-\frac{a}{2x_n}.\]
These comments deduce that there is no interval $Q\in\mathscr{Q}_a$ satisfies the following two conditions simultaneously
\begin{equation*}
\l\{
\begin{array}{l}
c_Q+\frac{am(c_Q)}{2}>A'_n=x_n+\frac{a}{2x_n} \\
c_Q-\frac{am(c_Q)}{2}<A_n=x_n-\frac{a}{2x_n}
\end{array}\r.
\end{equation*}
when $n\geq 2$. This finishes the proof of the claim.
\end{proof}

Now we turn to the proof of Theorem \ref{GM_2001_thm4.3}. Firstly, based on the radial versions of the local operators and weights constructed above, we can prove the following two lemmas by imitating the proofs in \cite{GM_2001} and using some basic facts on the Gaussian measure spaces.
\begin{lemma}\label{GM_2001_lemma4.2}
Let $0<a<\infty$ and $1\leq q\leq \frac{d}{d-\alpha}$. Then the pair of weights $(u,v)\in \mathscr{A}_{1,q,\alpha}^a$ if and only if
\begin{equation}\label{GM_2001_lemma4.2_1}
\l(M_{\beta}^au(x)\r)^{1/q}\leq C v(x)
\end{equation}
with $\beta=d-(d-\alpha)q$ for $\gamma'$-almost everywhere $x\in\bR^d$. Indeed, both conditions hold with the same constant.
\end{lemma}

\begin{proof}
Set
\[\widetilde{\mathscr{Q}}=\l\{Q\subset \bR^d: c_Q\in \bQ^d \;\;\text{and}\;\; \ell(Q)\in \bQ^{+}\r\}, \quad \widetilde{\mathscr{Q}}_a=\widetilde{\mathscr{Q}}\cap \mathscr{Q}_a.\]
By a continuity argument, it is enough to consider the cubes just in $\widetilde{\mathscr{Q}}_a$ in the definition of $M_{\alpha}^a$.

Suppose that $(u,v)\in \mathscr{A}_{1,q,\alpha}^a$ with constant $C_0$. For every $Q\in \widetilde{\mathscr{Q}}_a$, define
\[N(Q)=\l\{x\in Q: \frac{1}{\ell(Q)^{d-\alpha}}\l(\int_Q u(y)\ddd\gamma'(y)\r)^{1/q}>C_0 v(x)\r\}, \quad  N=\bigcup_{Q\in\widetilde{\mathscr{Q}}_a} N(Q).\]
Then by the definition of $\mathscr{A}_{1,q,\alpha}^a$ we have $\gamma'(N(Q))=0$ and, since $\widetilde{\mathscr{Q}}_a$ is countable, $\gamma'(N)=0$. Set
\[F=\l\{y\in\bR^d: \l(M_{\beta}^a(u)(y)\r)^{1/q}>C_0 v(y)\r\},\]
where $\frac{d-\beta}{q}=d-\alpha$. Hence for every $y\in F$, there exists a cube $Q\in\widetilde{\mathscr{Q}}_a(y)$ such that
\[\frac{1}{\ell(Q)^{d-\alpha}}\l(\int_Q u(x)\ddd\gamma'(x)\r)^{1/q}>C_0 v(y).\]
This yields that $F\subset N$ and $\gamma'(F)=0$. Equivalently, we have proved
\[\l(M_{\beta}^au(x)\r)^{1/q}\leq C_0 v(x) \quad \gamma'-a.e. \;\;x\in\bR^d.\]

On the other hand if (\ref{GM_2001_lemma4.2_1}) holds with constant $C_0$, then for any fixed $Q\in\widetilde{\mathscr{Q}}_a$ and $\gamma'$-a.e. $y\in Q$ we directly conclude
\begin{align*}
\frac{1}{\ell(Q)^{d-\alpha}}\l(\int_Q u(y)\ddd\gamma'(y)\r)^{1/q}&\leq \l(\sup_{Q\in\mathscr{Q}_a(y)} \frac{1}{\ell(Q)^{(d-\alpha)q}} \int_Q u(x)\ddd\gamma'(x)\r)^{1/q} \\
&=\l(M_{\beta}^a (u)(y)\r)^{1/q} \\
&\leq C_0 v(y).
\end{align*}
\end{proof}

\begin{lemma}\label{GM_2001_thm4.3.1}
Given $a\in(0,\infty)$, $1\leq p\leq q<\infty$, $0\leq\alpha<d$ and a pair of weights $(u,v)$, then the following two statements are equivalent:
\begin{itemize}
\item[(i)] $(u,v)\in \mathscr{A}_{p,q,\alpha}^{a}$.
\item[(ii)] For every $f\geq0$ and every cube $Q\in\mathscr{Q}_a$,
\[\l(\frac{1}{\ell(Q)^{d-\alpha}}\int_Q f(x)\ddd\gamma'(x)\r)^q \l(\int_Q u(x)\ddd\gamma'(x)\r)\leq C\l(\int_Q f(x)^p v(x)\ddd\gamma'(x)\r)^{q/p}.\]
\end{itemize}
\end{lemma}

\begin{proof}
We adapt some ideas from \cite{GM_2001}. Firstly, let us prove the conclusion (ii) with the assumption that $(u,v)\in \mathscr{A}_{p,q,\alpha}^{a}$. In the case $p=1$, for every $f\geq0$ and every $Q\in \mathscr{Q}_a$ we have
\begin{align*}
\l(\frac{1}{\ell(Q)^{d-\alpha}}\int_Q f(x)\ddd\gamma'(x)\r)^q \l(\int_Q u(x)\ddd\gamma'(x)\r)
&=\l[\int_Q f(x)\frac{1}{\ell(Q)^{d-\alpha}} \l(\int_Q u(x)\ddd\gamma'(x)\r)^{1/q}\ddd\gamma'(x)\r]^q \\
&\leq C \l(\int_Q f(x) v(x)\ddd\gamma'(x)\r),
\end{align*}
where the last inequality comes from $(u,v)\in \mathscr{A}_{1,q,\alpha}^a$. In the case $1<p<\infty$, by $\holder$'s inequality
\[\l(\frac{1}{\ell(Q)^{d-\alpha}}\int_Q f(x)\ddd\gamma'(x)\r)^q
\leq \frac{1}{\ell(Q)^{(d-\alpha)q}}\l(\int_Q f(x)^p v(x) \ddd\gamma'(x)\r)^{q/p} \l(\int_Q v(x)^{1-p'} \ddd \gamma'(x)\r)^{q/{p'}}.\]
Thus the condition $(u,v)\in\mathscr{A}_{p,q,\alpha}^a$ yields that
\begin{align*}
&\l(\frac{1}{\ell(Q)^{d-\alpha}}\int_Q f(x)\ddd\gamma'(x)\r)^q \l(\int_Q u(x)\ddd\gamma'(x)\r) \\
&\leq \frac{1}{\ell(Q)^{(d-\alpha)q}}\l(\int_Q f(x)^p v(x) \ddd\gamma'(x)\r)^{q/p} \l(\int_Q v(x)^{1-p'} \ddd \gamma'(x)\r)^{q/{p'}} \l(\int_Q u(x)\ddd\gamma'(x)\r) \\
&\leq [u,v]^q_{\mathscr{A}_{p,q,\alpha}^a} \l(\int_Q f(x)^p v(x) \ddd\gamma'(x)\r)^{q/p}.
\end{align*}

On the other hand, suppose the statement (ii) is right. Take $f\geq 0$. For any $S\subset Q\in\mathscr{Q}_a$, using (ii) with $f\chi_{S}(x)$ we get the following
\begin{equation}\label{GM_2001_thm4.3.1_1}
\l(\frac{1}{\ell(Q)^{d-\alpha}}\int_{S} f(x)\ddd \gamma'(x)\r)^q\l(\int_Q u(x)\ddd\gamma'(x)\r)\leq C\l(\int_S f(x)^p v(x)\ddd \gamma'(x)\r)^{q/p}.
\end{equation}
Then by taking $f\equiv1$ we obtain
\begin{equation}\label{GM_2001_thm4.3.1_2}
\l(\frac{\gamma'(S)}{\ell(Q)^{d-\alpha}}\r)^q\l[u(Q,\gamma')\r]\leq C \l[v(S,\gamma')\r]^{q/p},
\end{equation}
where
\[\gamma'(S)=\int_{S} \ddd \gamma'(x),\quad u(Q,\gamma')=\int_Q u(x)\ddd\gamma'(x), \quad v(S,\gamma')=\int_S v(x)\ddd \gamma'(x).\]
For any $Q\in\mathscr{Q}_a$, from this inequality (\ref{GM_2001_thm4.3.1_2}) we claim that
\[v(x)>0 \;\;\;\gamma'-a.e. \;x\in Q \quad\text{unless} \quad u(x)=0 \;\;\;\gamma'-a.e. \;x\in Q;\]
\[u(Q,\gamma')<\infty \quad \text{unless} \quad v(x)=\infty \;\;\;\gamma'- a.e. \;x\in Q.\]
Indeed, if $v(x)=0$ on some $S\subset Q$ with $\gamma'(S)>0$, from the inequality (\ref{GM_2001_thm4.3.1_2}) we have $u(Q,\gamma')=0$; if $u(Q,\gamma')=\infty$, using this estimate (\ref{GM_2001_thm4.3.1_2}) again, we conclude that $v(S,\gamma')=\infty$ for all $S\subset Q$ with $\gamma'(S)>0$. Hence we finish the proof of the claim.

Once we have done these observations shown above, we can prove this lemma by following the scheme of the proof in \cite[Theorem 4.3]{GM_2001}. For the case $1<p<\infty$, we can obtain the desired conclusion by setting
\[f(x)=v(x)^{1-p'}, \quad S_j=\l\{x\in Q: v(x)>\frac{1}{j}\r\}\]
where $Q\in\mathscr{Q}_a$ and using some locally integrable arguments together with the fact $v(x)>0$ for $\gamma'$-a.e. $x\in Q$; for the case $p=1$, we can get the desired result by rewriting the estimate (\ref{GM_2001_thm4.3.1_2}) as
\[\frac{1}{\ell(Q)^{d-\alpha}}\l(\int_Q u(x) \ddd \gamma'(x)\r)^{1/q}\leq C \frac{v(S,\gamma')}{\gamma'(S)} \quad \forall S\subset Q\in\mathscr{Q}_a \;\;\text{with}\;\; \gamma'(S)>0,\]
considering
\[a>\mathop{\mathrm{ess}\inf}\limits_{Q}:=\inf\l\{t>0: \gamma'\{x\in Q: v(x)<t\}>0\r\}, \quad S_a=\{x\in Q: v(x)<a\}\subset Q\in\mathscr{Q}_a\]
and using a continuity argument. This part is a little long but essentially the same as~\cite[Theorem 4.3]{GM_2001} since the restriction $Q\in\mathscr{Q}_a$ does not metter. Thereby we omit the detailed proof here.
\end{proof}

Next we can use the dyadic analysis and covering theorem together with some techniques on Gaussian measure spaces to establish the two-weight weak-type boundedness of $M_{\alpha}^a$ with respect to the $A_{p,q,\alpha}^a$ condition on Gaussian Lebesgue spaces.
\begin{proof}[\textbf{Proof of Theorem \ref{GM_2001_thm4.3}}]
By using Proposition \ref{DHY_RBAA}, we will prove that if $(u,v)\in\mathscr{A}_{p,q,\alpha}^b$ for some $b>a$, then
\begin{equation}\label{GM_2001_thm4.3_1}
\int_{\{x\in\bR^d: M_{\alpha}^a(f)(x)>\lambda\}} u(x)\ddd\gamma'(x)\lesssim \frac{1}{\lambda^q} \l(\int_{\bR^d}|f(x)|^p v(x)\ddd\gamma'(x)\r)^{q/p}
\end{equation}
holds for every $\lambda>0$.

It is sufficient to concentrate on the nonnegative functions $f\in L^p(\bR^d, v,\gamma')$. By Lemma \ref{GM_2001_thm4.3.1}, we know that the following estimate
\begin{equation}\label{GM_2001_thm4.3_2}
\l(\frac{1}{\ell(Q)^{d-\alpha}}\int_Q f(x)\ddd\gamma'(x)\r)^q \l(\int_Q u(x)\ddd\gamma'(x)\r)\lesssim\l(\int_Q f(x)^p v(x)\ddd\gamma'(x)\r)^{q/p}
\end{equation}
holds for every $f\geq0$ and every $Q\in\mathscr{Q}_b$. When $f\in L^p(\bR^d, v,\gamma')$ and $Q\in \mathscr{Q}_b$ with $u(Q,\gamma')>0$, the inequality (\ref{GM_2001_thm4.3_2}) yields that
\[\int_Q f(x)\ddd\gamma'(x)<\infty.\]
It follows that if $\int_{Q_0} f(x)\ddd\gamma'(x)=\infty$ for some cube $Q_0$, then $u(Q_0,\gamma')=0$ and there is a improper point $x_0$ of the function $f(x)$ contained in the closure of the cube $Q_0$. For this $x_0$, using the estimate (\ref{GM_2001_thm4.3_2}) again we conclude that $u(Q',\gamma')=0$ for all $Q'\in \mathscr{Q}_b(x_0)$. On the other hand, by the definition of $M_{\alpha}^a$ we konw that $M_{\alpha}^a(f)(x)=\infty$ for all $x\in Q'\in \mathscr{Q}_{a}(x_0)$.
Thus we observe that
\[u(x)\equiv0, \quad M_{\alpha}^a(f)(x)\equiv\infty\]
for $\gamma'$-a.e. $x\in\bigcup_{Q'\in \mathscr{Q}_a(x_0)} Q'$. Notice that the desired result (\ref{GM_2001_thm4.3_1}) also holds in this case. Thereby we only need to focus on the situation $x\notin\bigcup_{Q'\in \mathscr{Q}_b(x_0)} Q'$ where $M_{\alpha}^a(f)(x)<\infty$. In this way, without loss of generality, we can assume that
\[\int_Q f(x)\ddd\gamma'(x)<\infty\]
for all $Q\in \mathscr{Q}_a$. By the Heine-Borel theorem, we can assume $f\in L^1_{\mathrm{loc}}(\bR^d, \gamma')$. Moreover, by defining $f_k=f\chi_{Q(0,k)}$, if the estimate (\ref{GM_2001_thm4.3_1}) holds for each $f_k\in L^1(\bR^d, \gamma')$ with the constant independent of $k$, then using Fatou's lemma we can obtain (\ref{GM_2001_thm4.3_1}) for all $f\in L^1_{\mathrm{loc}}(\bR^d, \gamma')$. Hence we can assume $f\in L^1(\bR^d, \gamma')$. Take all these remarks into account, we only need to prove the desired result for
\[f\geq0, \quad f\in L^p(\bR^d, v, \gamma')\cap L^1(\bR^d, \gamma').\]
Define 
\[E_{\lambda}:=\{x\in\bR^d: M_{\alpha}^a(f)(x)>\lambda\}.\] Then for every $x\in E_{\lambda}$, there exists a cube $Q_x\in\mathscr{Q}_a(x)$ such that
\[\frac{1}{\ell(Q_x)^{d-\alpha}}\int_{Q_x} |f(y)|\ddd \gamma' (y)>\lambda.\]
Note that $\ell(Q_x)\leq a$. By \cite[Theorem 1.5]{Sawano_2005}, there exists a constant $N=N(d,k)$ depending only on the dimension $d$ and the ratio $k=b/a$ such that
\[E_{\lambda}\subset \bigcup_{x\in E_{\lambda}}Q_x \subset \bigcup_{j=1}^{N} \bigcup_{x_{\rho}\in E_{\lambda, j}} k Q_{x_{\rho}},\]
where the cubes in the subfamilies $\{Q_{x_{\rho}}\}_{x_{\rho}\in E_{\lambda, j}}$ are disjiont for fixed $j$.
Recall that (\ref{GM_2001_thm4.3_2}) leads to (\ref{GM_2001_thm4.3.1_1}). By taking $S=Q_{x_\rho}$ and $Q=kQ_{x_\rho}\in\mathscr{Q}_b$ in this estimate (\ref{GM_2001_thm4.3.1_1}) we can obtain
\begin{align*}
\int_{E_{\lambda}} u(x)\ddd\gamma'(x)&\leq \sum_{j=1}^N \sum_{x_{\rho} \in E_{\lambda, j}} \int_{kQ_{x_\rho}} u(x)\ddd\gamma'(x) \\
&\lesssim \sum_{j=1}^N \sum_{x_{\rho} \in E_{\lambda, j}} \l(\frac{1}{\ell(Q_{x_\rho})^{d-\alpha}} \int_{Q_{x_\rho}} f(x)\ddd\gamma'(x)\r)^{-q} \l(\int_{Q_{x_\rho}} f(x)^p v(x) \ddd\gamma'(x)\r)^{q/p} \\
&\leq \sum_{j=1}^N \sum_{x_{\rho} \in E_{\lambda, j}} \frac{1}{\lambda^q} \l(\int_{Q_{x_\rho}} f(x)^p v(x) \ddd\gamma'(x)\r)^{q/p} \\
&\leq \sum_{j=1}^N \frac{1}{\lambda^q} \l(\sum_{x_{\rho} \in E_{\lambda, j}} \int_{Q_{x_\rho}} f(x)^p v(x) \ddd\gamma'(x)\r)^{q/p} \\
&\leq \frac{N}{\lambda^q}\l(\int_{\bR^d} f(x)^p v(x)\ddd\gamma'(x)\r)^{q/p},
\end{align*}
where we have used the fact $p\leq q$ and the disjointness of $\{Q_{x_{\rho}}\}_{E_{\lambda, j}}$ for fixed $j$.
\end{proof}

Now we turn to the proof of the two-weight weak-type estimate for the local fractional integral operator $\tilde{I}_{\alpha}^a$ with cubes on Gaussian measure spaces. We need to introduce the following lemma first.
\begin{lemma}\label{LM_2020_lemma3.4}
Let $a\in (0,\infty)$, $0\leq \alpha_1<\alpha<\alpha_2\leq d$ and $f\in L_{\mathrm{loc}}^1(\gamma)$. If $\tilde{b}>a$, then there exists a constant $C$ independent of the function $f$ such that
\[|\tilde{I}_{\alpha}^{a}(f)(x)| \leq C\l(M_{\alpha_1}^a (f)(x)\r)^{\frac{\alpha_2-\alpha}{\alpha_2-\alpha_1}} \l(M_{\alpha_2}^{\tilde{b}} (f)(x)\r)^{\frac{\alpha-\alpha_1}{\alpha_2-\alpha_1}}\]
holds for all $x\in \bR^d$.
\end{lemma}
\begin{proof}
We adapt some ideas from \cite[Lemma 3.4]{LM_2020} and \cite[Lemma 4.1]{LX_2017}. Set $\dot{\alpha}_1=\alpha_1/d$ and $\dot{\alpha}_2=\alpha_2/d$. If $x=(x_1, \ldots, x_d) \in \bR^d$, we define the norm $\|\cdot\|$ on $\bR^d$ by
\[\|x\|=\max\{|x_1|,\ldots, |x_d|\}.\]
Then for any $x\in\bR^d$, we divide the proof into two cases.

\textbf{Case 1:} $\gamma\l[Q(x, am(x))\r]\leq \l(\frac{M_{\alpha_2}^{\tilde{b}} (f)(x)}{M_{\alpha_1}^a (f)(x)}\r)^{\frac{1}{\dot{\alpha}_2-\dot{\alpha}_1}}$.

By using the locally reverse doubling property of $\gamma$ on the admissiable cubes $\mathscr{Q}_a$, we obtain a constant $R=R_{d,a}\in (1,\infty)$ such that
\begin{align*}
|\tilde{I}_{\alpha}^a (f)(x)|&\leq \sum_{j=0}^{\infty} \int_{2^{-j-1} am(x)\leq 2\|x-y\|<2^{-j}am(x)} \frac{|f(y)|}{[\gamma(Q(x, |x-y|))]^{1-\dot{\alpha}}} \ddd\gamma(y) \\
&\lesssim \sum_{j=0}^{\infty} \frac{[\gamma(Q(x,2^{-j}am(x)))]^{\dot{\alpha}-\dot{\alpha}_1}}{[\gamma(Q(x,2^{-j}am(x)))]^{1-\dot{\alpha}_1}} \int_{Q(x, 2^{-j}am(x))} |f(y)| \ddd\gamma(y) \\
&\leq [\gamma(Q(x, am(x)))]^{\dot{\alpha}-\dot{\alpha}_1} \sum_{j=0}^{\infty} R^{-j(\dot{\alpha}-\dot{\alpha}_1)} M_{\alpha_1}^a (f)(x) \\
&\lesssim [\gamma(Q(x, am(x)))]^{\dot{\alpha}-\dot{\alpha}_1} M_{\alpha_1}^a (f)(x).
\end{align*}
Combining this estimate and the assumption in this case, we get the desired conclusion that
\begin{align*}
|\tilde{I}_{\alpha}^{a}(f)(x)|&\lesssim \l(M_{\alpha_1}^a (f)(x)\r)^{\frac{\dot{\alpha}_2-\dot{\alpha}}{\dot{\alpha}_2-\dot{\alpha}_1}} \l(M_{\alpha_2}^{\tilde{b}} (f)(x)\r)^{\frac{\dot{\alpha}-\dot{\alpha}_1}{\dot{\alpha}_2-\dot{\alpha}_1}} \\
&=\l(M_{\alpha_1}^a (f)(x)\r)^{\frac{\alpha_2-\alpha}{\alpha_2-\alpha_1}} \l(M_{\alpha_2}^{\tilde{b}} (f)(x)\r)^{\frac{\alpha-\alpha_1}{\alpha_2-\alpha_1}}.
\end{align*}

\textbf{Case 2:} $\gamma\l[Q(x, am(x))\r]> \l(\frac{M_{\alpha_2}^{\tilde{b}} (f)(x)}{M_{\alpha_1}^a (f)(x)}\r)^{\frac{1}{\dot{\alpha}_2-\dot{\alpha}_1}}$.

From \cite[Lemma 2.11]{LSY_2014} we know there exists an $r\in (0,a)$ such that
\begin{equation}\label{LM_2020_lemma3.4_1}
\frac{1}{2}\l(\frac{M_{\alpha_2}^{\tilde{b}} (f)(x)}{M_{\alpha_1}^a (f)(x)}\r)^{\frac{1}{\dot{\alpha}_2-\dot{\alpha}_1}}< \gamma\l[Q(x, rm(x))\r]< \l(\frac{M_{\alpha_2}^{\tilde{b}} (f)(x)}{M_{\alpha_1}^a (f)(x)}\r)^{\frac{1}{\dot{\alpha}_2-\dot{\alpha}_1}}.
\end{equation}
Then we can write the following
\[|\tilde{I}_{\alpha}^a(f)(x)|\leq \int_{2\|x-y\|<rm(x)} \frac{|f(y)|}{[\gamma(Q(x, |x-y|))]^{1-\dot{\alpha}}} \ddd\gamma(y) + \int_{rm(x)\leq 2\|x-y\|<am(x)} \cdots \ddd\gamma(y)=: I_1+I_2.\]
For the term $I_1$, using the similar method in the first case, except for the $a$ replaced by $r$, we conclude that
\[I_1\lesssim [\gamma(Q(x, rm(x)))]^{\dot{\alpha}-\dot{\alpha}_1} M_{\alpha_1}^a (f)(x).\]
For the term $I_2$, setting $N_r:=\lfloor\log_{\frac{\tilde{b}}{a}} \l(\frac{a}{r}\r)\rfloor$ be the maximal integer no more than $\log_{\frac{\tilde{b}}{a}} \l(\frac{a}{r}\r)$ we can get
\begin{align*}
I_2&\leq \sum_{j=0}^{N_r} \int_{\l(\frac{\tilde{b}}{a}\r)^j rm(x)\leq 2\|x-y\|<\l(\frac{\tilde{b}}{a}\r)^{j+1} rm(x)} \frac{|f(y)|}{[\gamma(Q(x, |x-y|))]^{1-\dot{\alpha}}} \ddd\gamma(y) \\
&\leq \sum_{j=0}^{N_r} \frac{1}{\l[\gamma\l(Q\l(x,\l(\frac{\tilde{b}}{a}\r)^{j}rm(x)\r)\r)\r]^{1-\dot{\alpha}}} \int_{Q\l(x,\l(\frac{\tilde{b}}{a}\r)^{j+1}rm(x)\r)} |f(y)| \ddd\gamma(y).
\end{align*}
Note that $Q\l(x,\l(\frac{\tilde{b}}{a}\r)^{j+1} rm(x)\r) \in \mathscr{Q}_{\tilde{b}}$ for all $j=1,\ldots, N_r$. Hence using the similar approach in the first case, by the locally reverse doubling property of $\gamma$ on $\mathscr{Q}_{\tilde{b}}$ we know there exists a constant $\tilde{R}=R_{d,\tilde{b}}\in (1,\infty)$ such that
\begin{align*}
I_2&\lesssim \sum_{j=0}^{N_r} \frac{\l[\gamma\l(Q\l(x, \l(\frac{\tilde{b}}{a}\r)^j rm(x)\r)\r)\r]^{\dot{\alpha}-\dot{\alpha}_2}}{\l[\gamma\l(Q\l(x, \l(\frac{\tilde{b}}{a}\r)^{j+1} rm(x)\r)\r)\r]^{1-\dot{\alpha}_2}} \int_{Q\l(x,\l(\frac{\tilde{b}}{a}\r)^{j+1} rm(x)\r)} |f(y)| \ddd\gamma(y) \\
&\leq \l[\gamma\l(Q\l(x, rm(x)\r)\r)\r]^{\dot{\alpha}-\dot{\alpha}_2} \sum_{j=0}^{N_r} \tilde{R}^{j(\dot{\alpha}-\dot{\alpha}_2)} M_{\alpha_2}^{\tilde{b}} (f)(x) \\
&\leq \l[\gamma\l(Q\l(x, rm(x)\r)\r)\r]^{\dot{\alpha}-\dot{\alpha}_2} \sum_{j=0}^{\infty} \tilde{R}^{j(\dot{\alpha}-\dot{\alpha}_2)} M_{\alpha_2}^{\tilde{b}} (f)(x) \\
&\lesssim  \l[\gamma\l(Q\l(x, rm(x)\r)\r)\r]^{\dot{\alpha}-\dot{\alpha}_2} M_{\alpha_2}^{\tilde{b}} (f)(x).
\end{align*}
Combining these two estimates of $I_1$ and $I_2$, then taking the condition (\ref{LM_2020_lemma3.4_1}) into account we conclude the following desired result
\begin{align*}
|\tilde{I}_{\alpha}^a(f)(x)|&\lesssim [\gamma(Q(x, rm(x)))]^{\dot{\alpha}-\dot{\alpha}_1} M_{\alpha_1}^a (f)(x)+ \l[\gamma\l(Q\l(x, rm(x)\r)\r)\r]^{\dot{\alpha}-\dot{\alpha}_2} M_{\alpha_2}^{\tilde{b}} (f)(x) \\
&\lesssim \l(M_{\alpha_1}^a (f)(x)\r)^{\frac{\alpha_2-\alpha}{\alpha_2-\alpha_1}} \l(M_{\alpha_2}^{\tilde{b}} (f)(x)\r)^{\frac{\alpha-\alpha_1}{\alpha_2-\alpha_1}}.
\end{align*}
Thereby we complete the proof of this lemma.
\end{proof}

\begin{proof}[\textbf{Proof of Theorem \ref{TWWTEFI}}]
Assume that $(u,v)\in A_{p,q,\alpha-\varepsilon}^b$ for some $b>a$ and $\varepsilon>0$. First, the fact $\gamma(Q)<1$ yields that
\[A_{p,q,\alpha_1}^{b}\subset A_{p,q,\alpha_2}^{b}\]
if $0<\alpha_1<\alpha_2<d$. Hence it is sufficient to consider the case $0<\varepsilon<\min\{\alpha, d-\alpha\}$. Then using Lemma \ref{LM_2020_lemma3.4} with
\[\tilde{b}=\frac{a+b}{2}, \quad \alpha_2-\alpha=\alpha-\alpha_1=\varepsilon\]
and using the $\holder$'s inequality for weak spaces (see \cite[Exercise 1.1.15]{GTM_249}) with $q_1=q_2=2q$ we conclude that
\begin{align*}
\l\|\tilde{I}_{\alpha}^a(f)\r\|_{L^{q,\infty}(\bR^d, u, \gamma)}
&\lesssim \l\|\l(M_{\alpha+\varepsilon}^{\frac{a+b}{2}}(f)(\cdot)\r)^{1/2}\l(M_{\alpha-\varepsilon}^a(f)(\cdot)\r)^{1/2}\r\|_{L^{q,\infty}(\bR^d, u, \gamma)} \\
&\leq \l\|\l(M_{\alpha+\varepsilon}^{\frac{a+b}{2}}(f)(\cdot)\r)^{1/2}\r\|_{L^{q_1,\infty}(\bR^d, u, \gamma)} \l\|\l(M_{\alpha-\varepsilon}^a(f)(\cdot)\r)^{1/2}\r\|_{L^{q_2,\infty}(\bR^d, u, \gamma)} \\
&=\l\|M_{\alpha+\varepsilon}^{\frac{a+b}{2}}(f)\r\|^{1/2}_{L^{q,\infty}(\bR^d, u, \gamma)} \l\|M_{\alpha-\varepsilon}^a(f)\r\|^{1/2}_{L^{q,\infty}(\bR^d, u, \gamma)} \\
&\lesssim \|f\|_{L^p(\bR^d, u, \gamma)},
\end{align*}
where the last inequality comes from Theorem \ref{GM_2001_thm4.3}.
\end{proof}

Based on the Theorem \ref{GM_2001_thm4.3} proved above and the one-weight results in~\cite{WZL_2016}, it is natural to ask the following question.
\begin{question}\label{DHY_2101_question1}
Is $(u,v)\in A_{p,q,\alpha}^a$ rather than $(u,v)\in\bigcup_{b'>a} A_{p,q,\alpha}^{b'}$ sufficient for the two-weight weak-type boundedness of $M_{\alpha}^a$?
\end{question}
It is not hard to see that $(u,v)\in A_{p,q,\alpha}^a$ is necessary for the two-weight weak-type boundedness of $M_{\alpha}^a$. This problem (the equivalence) may be solved by means of using other covering theorem instead of \cite[Theorem 1.5]{Sawano_2005} or other more delicate approach, but we have no progress yet. On the other hand, we have the following example to yield that the relation
\[\mathscr{A}_{p,q,\alpha}^a=\bigcup_{b'>a}\mathscr{A}_{p,q,\alpha}^{b'}\]
is not always true. Here we only give the example for the case $a=1$, the method is also valid for general $a>0$.

\begin{example}
Let $u(x)$ and $v(x)$ be even functions. When $x\in\bR^+$,
\[v(x)=\l\{\begin{array}{ll}
x^{\frac{1}{p'-1}}e^{\frac{\dot{\alpha}x^2}{1-p'}}, & 0<x<1, \\
e^{\frac{\dot{\alpha}x^2}{1-p'}}, & otherwise,
\end{array}\r.\]
and
\[u(x)=\l\{\begin{array}{ll}
0, & 0<x<1, \\
e^{\dot{\alpha}x^2}, &otherwise.
\end{array}\r.\]
Then for any $b'>1$ we have
\[[u,v]_{\mathscr{A}_{p,q,\alpha}^1}<\infty, \quad [u,v]_{\mathscr{A}_{p,q,\alpha}^{b'}}=\infty.\]
\end{example}
\begin{proof}
Let's first consider the following
\begin{align*}
[u,v]_{\mathscr{A}_{p,q,\alpha}^{b'}}&=\sup_{(c,c_1)\in\mathscr{Q}_{b'}}\frac{1}{|c_1-c|^{1-\alpha}} \l(\int_{c}^{c_1} u(x)\ddd\gamma'(x)\r)^{1/q}\l(\int_{c}^{c_1} v(x)^{1-p'}\ddd\gamma'(x)\r)^{1/{p'}} \\
&=\sup_{(c,c_1)\in\mathscr{Q}_{b'}}\frac{1}{|c_1-c|^{1-\alpha}} \l(\int_{c}^{c_1} u'(x)\ddd x\r)^{1/q}\l(\int_{c}^{c_1} \sigma'(x) \ddd x\r)^{1/{p'}}
\end{align*}
where $b'>1$ and on $\bR^{+}$
\[\sigma'(x)=\l\{\begin{array}{ll}
1/x, & 0<x<1, \\
1, & otherwise,
\end{array}\r.\quad
u'(x)=\l\{\begin{array}{ll}
0, & 0<x<1, \\
1, &otherwise.
\end{array}\r.\]
We take the interval $(c,c_1):=(0,c')$ where $1<c'<\min\{2,b'\}$. It's easy to see that $(0,c')\in\mathscr{Q}_{b'}$. Hence we conclude
\begin{align*}
[u,v]_{\mathscr{A}_{p,q,\alpha}^{b'}}&\geq \frac{1}{{c'}^{1-\alpha}}\l(\int_1^{c'} \ddd x\r)^{1/q}\l(\int_0^{c'} \frac{1}{x} \ddd x\r)^{1/{p'}} \\
&=\frac{1}{{c'}^{1-\alpha}} \l(c'-1\r)^{1/q} \l(\int_0^{c'} \frac{1}{x} \ddd x\r)^{1/{p'}} \\
&=\infty.
\end{align*}
On the other hand, since
\begin{align*}
[u,v]_{\mathscr{A}_{p,q,\alpha}^{1}}&=\sup_{(c,c_1)\in\mathscr{Q}_1}\frac{1}{|c_1-c|^{1-\alpha}} \l(\int_{c}^{c_1} u(x)\ddd\gamma'(x)\r)^{1/q}\l(\int_{c}^{c_1} v(x)^{1-p'}\ddd\gamma'(x)\r)^{1/{p'}} \\
&=\sup_{(c,c_1)\in\mathscr{Q}_1}\frac{1}{|c_1-c|^{1-\alpha}} \l(\int_{c}^{c_1} u'(x)\ddd x\r)^{1/q}\l(\int_{c}^{c_1} \sigma'(x) \ddd x\r)^{1/{p'}},
\end{align*}
we divide the proof into three cases as follows.

\textbf{Case 1:} If the interval $[c,c_1]\subset [-1,1]$, we have $[u,v]_{\mathscr{A}_{p,q,\alpha}^{1}}=0$ by the definition of $u(x)$;

\textbf{Case 2:} If $c>1$ or $c_1<-1$, we deduce that
\[[u,v]_{\mathscr{A}_{p,q,\alpha}^{1}}\leq\frac{(c_1-c)^{1/q+1/{p'}}}{|c_1-c|^{1-\alpha}}\leq (c_1-c)^{\alpha+1/q-1/p}\leq1;\]

\textbf{Case 3:} If $1\in(c,c_1)$ or $-1\in(c,c_1)$, we only need to prove the result under the situation $1\in(c,c_1)$ by the symmetry. Moreover, by a continuity argument it is enough to consider the behavior of $c_1\to 1$ with $c_1-c=1$. In this case, set $y=c_1-1$. Then it follows that
\begin{align*}
&\frac{1}{|c_1-c|^{1-\alpha}} \l(\int_{c}^{c_1} u(x)\ddd\gamma'(x)\r)^{1/q}\l(\int_{c}^{c_1} v(x)^{1-p'}\ddd\gamma'(x)\r)^{1/{p'}} \\
&=\frac{1}{|c_1-c|^{1-\alpha}} \l(\int_{1}^{1+y} \ddd x\r)^{1/q}\l(\int_{y}^{1} \frac{1}{x}\ddd x+\int_{1}^{1+y} \ddd x\r)^{1/{p'}} \\
&=y^{1/q}\l(-\ln{y}+y\r)^{1/{p'}} \to 0
\end{align*}
as $y\to 0$. Combining all these three cases, we get the desired result $[u,v]_{\mathscr{A}_{p,q,\alpha}^{1}}<\infty$.
\end{proof}

\section{The local Sawyer type weights}\label{DHY_2101_TLSTW}
In this section, we begin with a result about the strict inclusion relation on the local-$a$ testing condition. The similar result on the $\mathscr{A}_{p,q,\alpha}^a$ condition has been proved in Proposition \ref{WZL_2016_prop3.2_2}.
\begin{prop}\label{WZL_2016_prop3.2_1}
Let $0<a<b<\infty$, $0\leq\alpha<d$ and $1<p\leq q<\infty$. Then
\[\mathscr{M}_{p,q,\alpha}^b\subsetneqq\mathscr{M}_{p,q,\alpha}^a.\]
\end{prop}

\begin{proof}
It's trivial that $\mathscr{M}_{p,q,\alpha}^b\subset\mathscr{M}_{p,q,\alpha}^a$ and it is enough to focus on the situation $1/q-1/p+\alpha\geq0$ due to the fact $\mathscr{M}_{p,q,\alpha}^a =\emptyset$ when $1/q-1/p+\alpha<0$. The desired weights are analogous to that in Proposition \ref{WZL_2016_prop3.2_2}. On $\bR^{+}$, we set
\[u(x)=\l\{\begin{array}{ll}
e^{(\dot{\alpha}-1)q|x|^2+|x|^2}, & x\in (0,1), \\
n^{-kq(p-1)}e^{(\dot{\alpha}-1)q|x|^2+|x|^2}, & x\in \l(1+(n-1)a+\sum_{i=1}^{n-1}b_i, 1+(n-1)a+\sum_{i=1}^{n-1}b_i+\frac{b_n-a_n}{2}\r),\\
n^{-kq(p+1)}e^{(\dot{\alpha}-1)q|x|^2+|x|^2}, & x\in \l(1+(n-1)a+\sum_{i=1}^{n-1}b_i+\frac{b_n-a_n}{2}, 1+na+\sum_{i=1}^{n}b_i\r),
\end{array}\r.\]
with $n\geq1$ and then for $n\geq 2$ we set
\[v(x)=\l\{\begin{array}{ll}
e^{\frac{|x|^2}{1-p'}}, & x\in(0,1+b_1), \\
n^{-kp(p-1)}e^{\frac{|x|^2}{1-p'}}, & x\in \l(1+(n-2)a+\sum_{i=1}^{n-1}b_i, 1+(n-1)a+\sum_{i=1}^{n}b_i-\frac{b_n-a_n}{2}\r),\\
n^{kp(p-1)}e^{\frac{|x|^2}{1-p'}}, & x\in \l(1+(n-1)a+\sum_{i=1}^{n}b_i-\frac{b_n-a_n}{2}, 1+(n-1)a+\sum_{i=1}^{n}b_i\r),
\end{array}\r.\]
where $a_n, b_n$ satisfy the same condition (\ref{condition_1}). Then by even extension, we can get $u(x)$ and $v(x)$ on $\bR^1$. For similar reasons we consider the following expression
\begin{equation}\label{WZL_2016_prop3.2_1.1}
\sup_{Q\in\mathscr{Q}_a}\l[\int_Q\l(\sup_{Q'\in\mathscr{Q}_a(x)}\frac{1}{\ell(Q')^{1-\alpha}}\int_{Q'\cap Q} v(y)^{1-p'}\ddd\gamma'(y)\r)^q u(x)\ddd\gamma(x)\r]^{\frac{1}{q}}\l(\int_Q v(x)^{1-p'}\ddd\gamma(x)\r)^{-\frac{1}{p}}
\end{equation}
only on $\bR^{+}$. Using the same notations as in Proposition \ref{WZL_2016_prop3.2_2}, we have
\begin{equation}\label{DHY_2101_equ_1}
u(x)\leq e^{(\dot{\alpha}-1)q|x|^2+|x|^2},\quad v(x)^{1-p'}\leq e^{|x|^2}
\end{equation}
on $(0,1+b_1+a)$ and on $\bR^1\setminus \{\cup Q_{n_1}\cup Q_{n_2}\}=:S$. Hence if $Q\subset S$ or if $Q\cap Q_1\neq\emptyset$ which means $Q\subset (0,1+b_1+a)$, we conclude that
\begin{align*}
&\l[\int_Q\l(\sup_{Q'\in\mathscr{Q}_a(x)}\frac{1}{\ell(Q')^{1-\alpha}}\int_{Q'\cap Q} v(y)^{1-p'}\ddd\gamma'(y)\r)^q u(x)\ddd\gamma(x)\r]^{\frac{1}{q}}\l(\int_Q v(x)^{1-p'}\ddd\gamma(x)\r)^{-\frac{1}{p}} \\
&\lesssim \l[\l(\frac{\ell(Q'\cap Q)}{\ell(Q')^{1-\alpha}}\r)^q \ell(Q)\r]^{1/q} \ell(Q)^{-1/p} \\
&\leq \frac{\ell(Q'\cap Q)}{\ell(Q')^{1-\alpha}} \ell(Q)^{1/q-1/p} \\
&\leq \ell(Q'\cap Q)^{\alpha}\ell(Q)^{1/q-1/p} \leq \ell(Q)^{\alpha+1/q-1/p} \leq a^{\alpha+1/q-1/p}
\end{align*}
by the estimates (\ref{DHY_2101_equ_1}). Thereby for the expression (\ref{WZL_2016_prop3.2_1.1}), we only need to focus on the case $Q\cap Q_n\neq\emptyset$ with $n\geq 2$ and furthermore the situations $Q\cap Q_{n_1}\neq \emptyset$ or $Q\cap Q_{n_2}\neq\emptyset$ where $n\geq2$.

If $Q\cap Q_{n_1}\neq \emptyset$ with $n\geq2$, then $Q'\in \mathscr{Q}_a$ yields that $Q'\cap Q_{n_2}=\emptyset$ and
\begin{equation}\label{WZL_2016_prop3.2_1.2}
\l(\frac{1}{\ell(Q')^{1-\alpha}}\int_{Q'\cap Q} v(y)^{1-p'}\ddd\gamma'(y)\r)^q \sim_{a} \l(n^{kp}e^{(1-\dot{\alpha})|x|^2} \frac{\ell(Q'\cap Q)}{\ell(Q')^{1-\alpha}}\r)^q.
\end{equation}
If $Q\cap Q_{n_2}\neq\emptyset$ with $n\geq2$, from $n^{-kp(p-1)}\leq n^{kp(p-1)}$ and $1-p'\leq 0$ we similarly have
\begin{equation}\label{WZL_2016_prop3.2_1.3}
\l(\frac{1}{\ell(Q')^{1-\alpha}}\int_{Q'\cap Q} v(y)^{1-p'}\ddd\gamma'(y)\r)^q \lesssim_{a} \l(n^{kp}e^{(1-\dot{\alpha})|x|^2}\frac{\ell(Q'\cap Q)}{\ell(Q')^{1-\alpha}}\r)^q.
\end{equation}
Based on these two estimates above, when $n\geq2$, for every $Q\in \mathscr{Q}_a$ we divide the proof into two cases. Firstly, if $Q\cap Q_{n_1}\neq \emptyset$, by the estimate (\ref{WZL_2016_prop3.2_1.2}) we have
\begin{align*}
&\l[\int_Q\l(\sup_{Q'\in\mathscr{Q}_a(x)}\frac{1}{\ell(Q')^{1-\alpha}}\int_{Q'\cap Q} v(y)^{1-p'}\ddd\gamma'(y)\r)^q u(x)\ddd\gamma(x)\r]^{1/q} \\
&\lesssim_{a} \l[\int_{Q\cap Q_{n_1}}\l(n^{kp}e^{(1-\dot{\alpha})|x|^2}\frac{\ell(Q'\cap Q)}{\ell(Q')^{1-\alpha}}\r)^q e^{(\dot{\alpha}-1)q |x|^2+|x|^2} n^{-kq(p-1)}e^{-|x|^2}\ddd x\r]^{1/q}  \\
&= n^k \frac{\ell(Q'\cap Q)}{\ell(Q')^{1-\alpha}} |Q\cap Q_{n_1}|^{1/q} \leq n^k \frac{\ell(Q'\cap Q)}{\ell(Q')^{1-\alpha}}\ell(Q)^{1/q}.
\end{align*}
Then from the assumption $1/q-1/p+\alpha\geq0$ we directly get
\begin{align*}
[u,v]_{\mathscr{M}_{p,q,\alpha}^a}&\lesssim_{a} n^k \frac{\ell(Q'\cap Q)}{\ell(Q')^{1-\alpha}}\ell(Q)^{1/q} \cdot \l(\int_Q v(x)^{1-p'}\ddd\gamma(x)\r)^{-1/p} \\
&\sim n^k \frac{\ell(Q'\cap Q)}{\ell(Q')^{1-\alpha}}\ell(Q)^{1/q} \cdot n^{-k} |Q|^{-1/p} \\
&= \frac{\ell(Q'\cap Q)}{\ell(Q')^{1-\alpha}}\ell(Q)^{1/q-1/p} \leq a^{\alpha+1/q-1/p}.
\end{align*}
Secondly, if $Q\cap Q_{n_2}\neq \emptyset$, then the inequality (\ref{WZL_2016_prop3.2_1.3}) holds. Noticing that $u(x)=n^{-kq(p+1)}e^{(\dot{\alpha}-1)q|x|^2+|x|^2}$ on $Q_{n_2}$, we can use an analogous method as in the first case to obtain
\begin{align*}
[u,v]_{\mathscr{M}_{p,q,\alpha}^a}&\lesssim_{a} n^{-k} \frac{\ell(Q'\cap Q)}{\ell(Q')^{1-\alpha}} \ell(Q)^{1/q} \cdot \l(\int_Q v(x)^{1-p'}\ddd\gamma(x)\r)^{-1/p} \\
&\lesssim_{a} n^{-k} \frac{\ell(Q'\cap Q)}{\ell(Q')^{1-\alpha}} \ell(Q)^{1/q} \cdot n^{k} |Q|^{-1/p} \\
&\leq a^{\alpha+1/q-1/p}.
\end{align*}
Combining these two cases together with the former cases $Q\cap Q_1\neq\emptyset$ and $Q\subset S$, we have proved the result
\[[u,v]_{\mathscr{M}_{p,q,\alpha}^a}\lesssim a^{\alpha+1/q-1/p}<\infty.\]
Next we consider the term $[u,v]_{\mathscr{M}_{p,q,\alpha}^b}$. By taking $Q=Q_n\in \mathscr{Q}_b$ and $Q'=Q_{n_1}\in \mathscr{Q}_b(y)$ for $y\in Q_{n_1}$, similar to the proof in Proposition \ref{WZL_2016_prop3.2_2}, we deduce the following
\begin{align*}
[u,v]_{\mathscr{m}_{p,q,\alpha}^b}&\geq\l[\int_{Q_{n_1}}\l(\frac{1}{\ell(Q_{n_1})^{d-\alpha}}\int_{Q_{n_1}} v(y)^{1-p'}\ddd \gamma'(y)\r)^q u(x)\ddd \gamma(x)\r]^{1/q}\cdot \l(\int_{Q_{n_2}}v(x)^{1-p'}\ddd\gamma(x)\r)^{-1/p} \\
&\sim_{b} \l(\int_{Q_{n_1}}n^{kpq}e^{(1-\dot{\alpha})q|x|^2}\ell(Q_{n_1})^{\alpha q}\cdot n^{-kq(p-1)}e^{(\dot{\alpha}-1)q|x|^2+|x|^2}e^{-|x|^2}\ddd x\r)^{1/q}\cdot n^k|Q_{n_2}|^{-1/p} \\
&=n^{2k}\ell(Q_{n_1})^{\alpha}\cdot |Q_{n_1}|^{1/q} |Q_{n_2}|^{-1/p} \\
&\geq \l(\frac{b-a}{2}\r)^{\alpha+1/q-1/p} n^{2k} (1+na+nb)^{-\alpha-1/q+1/p},
\end{align*}
where we have used the facts that $\alpha+1/q-1/p\geq0$ and
\[|x_n|\leq 1+na+\sum_{i=1}^n b_i\leq 1+na+nb.\]
Finally taking $k=\alpha+1/q-1/p+1>0$, we get the desired weights.
\end{proof}

Then we consider the two-weight strong-type estimate for $M_{\alpha}^a$ with respect to the local-$a$ testing condition on Gaussian Lebesgue spaces. To use the dyadic analysis in this situation, we need the following lemma first.
\begin{lemma}\label{GM_2001_lemma3.2}
Consider $a\in(0,\infty)$, $\alpha\in[0,d)$, $f\in L_{\mathrm{loc}}^1(\bR^d, \gamma')$ and $f\geq0$. If for some cube $Q\in\mathscr{Q}_a$ and for some $t>0$ we have
\[\frac{1}{\ell(Q)^{d-\alpha}}\int_Q f(y) \ddd\gamma'(y)>t,\]
then there exists a dyadic cube $P\in \mathscr{Q}_{2a+3\sqrt{d}a^2}$ such that $Q\subset 3P$ and
\[\frac{1}{\ell(P)^{d-\alpha}}\int_{P} f(y)\ddd\gamma'(y)>2^{\alpha-2d}t.\]
\end{lemma}

\begin{proof}
Take $k\in\bZ$ such that $2^{k-1}\leq \ell(Q)<2^k$. Then there exist at most $2^d$ dyadic cubes $\{P_i\}_{i=1}^N$ which have the following two properties:
\[\ell(P_i)= 2^k \quad \text{and} \quad Q\cap P_i\neq \emptyset.\]
From $\ell(Q)<\ell(P_i)\leq 2\ell(Q)$ and $Q\cap P_i\neq \emptyset$ we conclude that $Q\subset 3P_i$. Moreover, it is easy to see that at least one of these dyadic cubes, say $P$, satisfies the condition
\[\frac{\int_P f(y) \ddd\gamma'(y)}{\ell(Q)^{d-\alpha}}> \frac{t}{2^d}.\]
As a consequence, for this dyadic cube $P$ we have
\[\frac{1}{\ell(P)^{d-\alpha}}\int_P f(y) \ddd\gamma'(x)>\frac{\ell(Q)^{d-\alpha}t}{\ell(P)^{d-\alpha}2^d}\geq2^{\alpha-2d} t.\]
Then it remains to prove that $P\in \mathscr{Q}_{2a+3\sqrt{d}a^2}$. We divide the proof into four cases as follows.

\textbf{Case 1:} $|c_{Q}|\leq 1$ and $|c_P|\leq 1$, then $\ell(Q)\leq a$ and $\ell(P)\leq2\ell(Q)\leq2a$;

\textbf{Case 2:} $|c_{Q}|\leq 1$ and $|c_P|> 1$, then $\ell(Q)\leq a$ and
\[\ell(P)\cdot|c_P|\leq 2\ell(Q)\cdot\l(|c_Q|+\frac{3\sqrt{d}\ell(Q)}{2}\r)=2\ell(Q)|c_Q|+2\ell(Q)\frac{3\sqrt{d}\ell(Q)}{2}\leq 2a+3\sqrt{d} a^2;\]

\textbf{Case 3:} $|c_{Q}|> 1$ and $|c_P|\leq 1$, then $\ell(Q)\leq a/|c_Q|\leq a$ and $\ell(P)\leq2\ell(Q)\leq2a$;

\textbf{Case 4:} $|c_{Q}|> 1$ and $|c_P|> 1$, then $\ell(Q)\leq a/|c_Q|\leq a$ and
\[\ell(P)\cdot|c_P|\leq 2\ell(Q)\cdot\l(|c_Q|+\frac{3\sqrt{d}\ell(Q)}{2}\r)=2\ell(Q)|c_Q|+2\ell(Q)\frac{3\sqrt{d}\ell(Q)}{2}\leq 2a+3\sqrt{d} a^2.\]
Combining all these four cases, we obtain the desired conclusion $P\in \mathscr{Q}_{2a+3\sqrt{d}a^2}$.
\end{proof}

On the basis of this Lemma \ref{GM_2001_lemma3.2}, by following the scheme of the proof in \cite[Theorem 3.1]{GM_2001} and using the radialization method mentioned before, we prove the two-weighted strong-type boundedness of the local fractional maximal operator $M_{\alpha}^a$ on Gaussian Lebesgue spaces.
\begin{proof}[\textbf{Proof of Theorem \ref{GM_2001_thm3.1}}]
Without loss of generality, we can assume that $f\in L^p(\bR^d, v,\gamma)$ is a nonnegative bounded function with compact support. Thereby we conclude that $M_{\alpha}^a(f)(x)$ is finite $\gamma$-almost everywhere. Decompose $\bR^d$ by
\[\bR^d=\bigcup_{k\in\bZ} \Omega_k, \quad \Omega_k=\{x\in\bR^d: 2^k<M_{\alpha}^a(f)(x)\leq 2^{k+1}\}.\]
Then for every $k\in\bZ$ and every $x\in\Omega_k$, there exists a cube $Q_{x}^k\in\mathscr{Q}_a$ satisfying the condition
\[\frac{1}{\ell(Q_{x}^k)^{d-\alpha}}\int_{Q_{x}^k} f(y)\ddd\gamma'(y)>2^k.\]
By Lemma \ref{GM_2001_lemma3.2} we get a dyadic cube $P_{x}^k\in\mathscr{Q}_{2a+3\sqrt{d}a^2}$ such that $Q_{x}^k\subset 3P_{x}^k$ and
\begin{equation}\label{GM_2001_thm3.1_1}
\frac{1}{\ell(P_{x}^k)^{d-\alpha}}\int_{P_{x}^k} f(y)\ddd\gamma'(y)>2^{\alpha-2d}2^k.
\end{equation}
Notice that $\ell(P_{x}^k)\leq 2a+3\sqrt{d}a^2$. Hence for every fixed $k$, there is a subcollection of maximal disjoint dyadic cubes $\{P_j^k\}_j$ in the sense that for any $Q_{x}^k$ there exists a $P_j^k$ satisfying $Q_{x}^k\subset 3P_{j}^k$. This construction yields that $\Omega_k\subset \bigcup_{j} 3P_{j}^k$. Thus we can decompose $\Omega_k$ by defining
\[E_{1}^k=3P_{1}^k\bigcap \Omega_k, E_2^k=\l(3P_2^k\setminus 3P_1^k\r)\bigcap \Omega_k,\ldots, E_j^k=\l(3P_j^k\setminus \bigcup_{r=1}^{j-1} 3P_r^k\r)\bigcap \Omega_k, \ldots.\]
In conclusion we obtain
\[\bR^d=\bigcup_{k\in\bZ} \Omega_k=\bigcup_{j,k} E_j^k\]
where the sets $E_j^k$ are disjoint and $\Omega_k$ are disjoint. For a fixed integer $K>0$, define
\[\Lambda_K=\{(j,k)\in \bN\times\bZ: |k|\leq K\}.\]
By using the fact $E_j^k\subset\Omega_k$ and the condition (\ref{GM_2001_thm3.1_1}) we deduce
\begin{align*}
I_K&:=\int_{\bigcup_{k=-K}^{K} \Omega_k} \l(M_{\alpha}^a(f)(x)\r)^q u(x)\ddd\gamma(x)=\sum_{(j,k)\in\Lambda_K} \int_{E_j^k} \l(M_{\alpha}^a(f)(x)\r)^q u(x)\ddd\gamma(x) \\
&\leq \sum_{(j,k)\in\Lambda_K} \int_{E_j^k} u(x)\ddd\gamma(x) 2^{(k+1)q} \\
&\leq 2^{(2d-\alpha+1)q} \sum_{(j,k)\in\Lambda_K} \int_{E_j^k} u(x)\ddd\gamma(x)\l(\frac{1}{\ell(P_{j}^k)^{d-\alpha}}\int_{P_{j}^k} f(y)\ddd\gamma'(y)\r)^q \\
&= 2^{(2d-\alpha+1)q}3^{(d-\alpha)q} \sum_{(j,k)\in\Lambda_K} \int_{E_j^k} u(x)\ddd\gamma(x) \l(\frac{1}{\ell(3P_{j}^k)^{d-\alpha}}\int_{3P_{j}^k} \sigma(y)\ddd\gamma'(y)\r)^q \l(\frac{\int_{P_{j}^k} (f\sigma^{-1})(y)\sigma(y)\ddd\gamma'(y)}{\int_{3P_{j}^k} \sigma(y)\ddd\gamma'(y)}\r)^q \\
&=2^{(2d-\alpha+1)q}3^{(d-\alpha)q} \int_{\mathscr{Y}} T_K(f\sigma^{-1})^q \ddd\nu,
\end{align*}
where $\mathscr{Y}=\bN\times\bZ$, $\sigma(x)=v(x)^{1-p'}$. Furthermore, the measure $\nu$ in $\mathscr{Y}$ is given by
\[\nu(j,k):=\int_{E_j^k} u(x)\ddd\gamma(x) \l(\frac{1}{\ell(3P_{j}^k)^{d-\alpha}}\int_{3P_{j}^k} \sigma(y)\ddd\gamma'(y)\r)^q\]
and the operator $T_K$ for measurable function $h$ is defined by the following three equivalent forms
\begin{align*}
T_K(h)(j,k)&:=\frac{\int_{P_{j}^k} h(y)\sigma(y)\ddd\gamma'(y)}{\int_{3P_{j}^k} \sigma(y)\ddd\gamma'(y)}\chi_{\Lambda_K}(j,k) \\
&:\sim_{a,d} \frac{\int_{P_{j}^k} h(y)\sigma(y)\ddd y}{\int_{3P_{j}^k} \sigma(y)\ddd y}\chi_{\Lambda_K}(j,k) \\
&:\sim_{a,d} \frac{\int_{P_{j}^k} h(y)\sigma(y)\ddd\gamma(y)}{\int_{3P_{j}^k} \sigma(y)\ddd\gamma(y)}\chi_{\Lambda_K}(j,k)
\end{align*}
due to the fact that $P_{j}^k\in\mathscr{Q}_{2a+3\sqrt{d}a^2}$. In this way, if we can prove that the operator $T_K$ is bounded from $L^p(\bR^d,\sigma, \gamma)$ to $L^q(\mathscr{Y},\nu)$ independently of $K$, we can obtain
\[I_K\lesssim \int_{\mathscr{Y}} T_K(f\sigma^{-1})^q \ddd\nu \lesssim \l(\int_{\bR^d} (f\sigma^{-1})^p \sigma\ddd\gamma\r)^{q/p}= \l(\int_{\bR^d} f(x)^p v(x)\ddd\gamma(x)\r)^{q/p}.\]
Based on the uniformity of $K$ and the monotone convergence theorem, we shall prove the desired result by letting $K\to\infty$.

Consequently, it remains to prove the uniform boundedness of $T_K$. It is easy to see that
\[T_K: L^{\infty}(\bR^d,\sigma,\gamma)\to L^{\infty}(\mathscr{Y},\nu)\]
with norm equal or less than $1$. Marcinkiewicz interpolation theorem yields that it is sufficient to show
\[T_K: L^{1}(\bR^d,\sigma,\gamma)\to L^{{q}/{p},\infty}(\mathscr{Y},\nu)\]
with norm independent of $K$. In other words, it is sufficient to prove that
\[\nu\l\{(j,k)\in\mathscr{Y}: T_K(h)(j,k)>\lambda\r\}\lesssim \l(\frac{1}{\lambda}\int_{\bR^d} |h(x)|\sigma\ddd\gamma(x)\r)^{q/p}\]
holds for all $\lambda>0$. To establish this estimate, for the fixed nonnegative bounded function $h(x)$ with compact support, set
\[F_{\lambda}=\l\{(j,k)\in\mathscr{Y}: T_K(h)(j,k)>\lambda\r\}=\l\{(j,k)\in\Lambda_K: T_K(h)(j,k)>\lambda\r\}.\]
Noticing the fact $E_j^k\subset 3P_j^k \in\mathscr{Q}_{6a+9\sqrt{d}a^2}$, we conclude
\begin{align*}
\nu(F_{\lambda})&=\sum_{(j,k)\in F_{\lambda}} \int_{E_j^k} u(x)\ddd\gamma(x) \l(\frac{1}{\ell(3P_{j}^k)^{d-\alpha}}\int_{3P_{j}^k} \sigma(y)\ddd\gamma'(y)\r)^q \\
&=\sum_{(j,k)\in F_{\lambda}} \int_{E_j^k} \l(\frac{1}{\ell(3P_{j}^k)^{d-\alpha}}\int_{3P_{j}^k} \sigma(y)\ddd\gamma'(y)\r)^q u(x) \ddd\gamma(x) \\
&\leq \sum_{(j,k)\in F_{\lambda}} \int_{E_j^k} \l(M_{\alpha}^{6a+9\sqrt{d}a^2}\l(\sigma \chi_{3P_j^k}\r)(x)\r)^q u(x) \ddd\gamma(x).
\end{align*}
Taking $\ell(P_j^k)\leq 2a+3\sqrt{d}a^2$ into account, we can extract a maximal disjoint subcollection $\{P_i\}_i$ from the collection $\{P_j^k: (j,k)\in F_{\lambda}\}$ in the sense that for every $(j,k)\in F_{\lambda}$ there exists an $i$ such that $P_{j}^k\subset P_i$. By the disjointness of $E_j^k$, the construction of $P_i$ and the fact $(u,v)\in \mathscr{M}_{p,q,\alpha}^{6a+9\sqrt{d}a^2}$ we can obtain
\begin{align*}
\nu(F_{\lambda})&\leq \sum_i\sum_{P_j^k\subset P_i} \int_{E_j^k} \l(M_{\alpha}^{6a+9\sqrt{d}a^2}\l(\sigma \chi_{3P_j^k}\r)(x)\r)^q u(x) \ddd\gamma(x) \\
&\leq \sum_i \int_{3P_i} \l(M_{\alpha}^{6a+9\sqrt{d}a^2}\l(\sigma \chi_{3P_i}\r)(x)\r)^q u(x) \ddd\gamma(x) \\
&\lesssim \sum_i\l(\int_{3P_i}\sigma(x) \ddd\gamma(x)\r)^{q/p}.
\end{align*}
Recall that these cubes $P_i$ were extracted from $\l\{P_j^k: (j,k)\in F_{\lambda}\r\}$ where $T_K(h)(j,k)>\lambda$. This yields
\[\int_{3P_i} \sigma(x)\ddd\gamma(x)<\frac{1}{\lambda}\int_{P_i}h(x)\sigma(x)\ddd\gamma(x).\]
Then by using $p\leq q$ and the disjointness of $P_i$ we shall conclude
\begin{align*}
\nu(F_{\lambda})&\lesssim \sum_i \l(\frac{1}{\lambda}\int_{P_i}h(x)\sigma(x)\ddd\gamma(x)\r)^{q/p} \\
&\leq \l(\sum_i \frac{1}{\lambda}\int_{P_i}h(x)\sigma(x)\ddd\gamma(x)\r)^{q/p} \\
&\leq \l(\frac{1}{\lambda}\int_{\bR^d}h(x)\sigma(x)\ddd\gamma(x)\r)^{q/p}.
\end{align*}
All these estimates above are independent of $K$, in other words, we have completed the proof.
\end{proof}
\begin{proof}[\textbf{Proof of Theorem \ref{GM_2001_thm6.5}}]
Assume that $(u,v)\in \mathscr{M}_{p,q,\alpha-\varepsilon}^{b}$ for some $b>6a+9\sqrt{d}a^2$ and $\varepsilon>0$. First, the fact $\gamma(Q)<1$ yields that
\[\mathscr{M}_{p,q,\alpha_1}^{b}\subset \mathscr{M}_{p,q,\alpha_2}^{b}\]
if $0<\alpha_1<\alpha_2<d$. Hence it is sufficient to consider the case $0<\varepsilon<\min\{\alpha, d-\alpha\}$. Since $b>6a+9\sqrt{d}a^2$, we can choose $\tilde{a}=\tilde{a}(a,b)>a$ such that $6\tilde{a}+9\sqrt{d}\tilde{a}^2<b$. For example, we can take
\[\tilde{a}=\frac{\frac{ab}{6a+9\sqrt{d}a^2}+a}{2}.\]
Then using Lemma \ref{LM_2020_lemma3.4} with
\[\tilde{b}=\tilde{a}, \quad \alpha_2-\alpha=\alpha-\alpha_1=\varepsilon\]
and using the $\holder$'s inequality with $q_1=q_2=2q$ we conclude that
\begin{align*}
\l\|\tilde{I}_{\alpha}^a(f)\r\|_{L^{q}(\bR^d, u, \gamma)}
&\lesssim \l\|\l(M_{\alpha+\varepsilon}^{\tilde{a}}(f)(\cdot)\r)^{1/2}\l(M_{\alpha-\varepsilon}^a(f)(\cdot)\r)^{1/2}\r\|_{L^{q}(\bR^d, u, \gamma)} \\
&\leq \l\|\l(M_{\alpha+\varepsilon}^{\tilde{a}}(f)(\cdot)\r)^{1/2}\r\|_{L^{q_1}(\bR^d, u, \gamma)} \l\|\l(M_{\alpha-\varepsilon}^a(f)(\cdot)\r)^{1/2}\r\|_{L^{q_2}(\bR^d, u, \gamma)} \\
&=\l\|M_{\alpha+\varepsilon}^{\tilde{a}}(f)\r\|^{1/2}_{L^{q}(\bR^d, u, \gamma)} \l\|M_{\alpha-\varepsilon}^a(f)\r\|^{1/2}_{L^{q}(\bR^d, u, \gamma)} \\
&\lesssim \|f\|_{L^p(\bR^d, u, \gamma)},
\end{align*}
where the last inequality comes from Theorem \ref{GM_2001_thm3.1}.
\end{proof}
As we have discussed in Question \ref{DHY_2101_question1}, there is also a similar problem for the local Sawyer type weights. We state the question as follows.
\begin{question}\label{DHY_2101_question2}
Is $(u,v)\in\mathscr{M}_{p,q,\alpha}^a$ rather than $(u,v)\in\mathscr{M}_{p,q,\alpha}^{6a+9\sqrt{d}a^2}$ sufficient for the two-weight strong-type boundedness of $M_{\alpha}^a$?
\end{question}
It is easy to see that $(u,v)\in\mathscr{M}_{p,q,\alpha}^a$ is necessary for the two-weight strong-type boundedness of $M_{\alpha}^a$. The parameter $6a+9\sqrt{d}a^2$ in our theorem may be smaller by finer estimates. We expect it can be $a$ but have no idea for this yet.

\bigskip\bigskip

\subsection*{Acknowledgements} This work was supported by National Natural Science Foundation of China (Grant Nos. 11871452 and 12071473), Beijing Information Science and Technology University Foundation (Grant Nos. 2025031).

\bigskip\bigskip

\vspace{-0.4cm}{\footnotesize

\begin{flushleft}
\vspace{0.3cm}\textsc{Boning Di\\
School of Mathematical Sciences\\
University of Chinese Academy of Sciences\\
Beijing, 100049\\
P. R. China}
	
\vspace{0.3cm}\textsc{Qianjun He\\
School of Applied Science\\
Beijing Information Science and Technology University\\
Beijing, 100192\\
P. R. China}
	
\emph{E-mail address}: \textsf{heqianjun16@mails.ucas.ac.cn}
	
\vspace{0.3cm}\textsc{Dunyan Yan\\
School of Mathematical Sciences\\
University of Chinese Academy of Sciences\\
Beijing, 100049\\
P. R. China}
\end{flushleft}

\end{document}